\documentclass[10pt]{amsart}
\usepackage{amssymb}

\newtheorem{thm}{Theorem}[section]
\newtheorem{lemma}[thm]{Lemma}
\newtheorem{proposition}[thm]{Proposition}

\newtheorem{definition}[thm]{Definition}

\newcommand{\p}{\mathbb{P}}
\newcommand{\q}{\mathbb{Q}}

\newcommand{\ot}{\mathrm{ot}}
\newcommand{\cf}{\mathrm{cf}}
\newcommand{\cof}{\mathrm{cof}}

\begin{document}

\title{Coherent Adequate Sets and Forcing Square}

\author{John Krueger}

\address{Department of Mathematics \\ 
University of North Texas \\
1155 Union Circle \#311430 \\
Denton, TX 76203}

\email{jkrueger@unt.edu}

\date{}


\thanks{ 
Coherent adequate set. 
Models as side conditions. Square sequence.}

\begin{abstract}
We introduce the idea of a coherent adequate set of models, which can be 
used as side conditions in forcing. 
As an application we define a forcing poset which adds a square sequence on 
$\omega_2$ using finite conditions.
\end{abstract}

\maketitle

In previous work \cite{krueger1} we introduced the idea of an adequate set of 
models and showed how to use adequate sets as side conditions in forcing 
with finite conditions. 
We gave several examples of forcing with adequate sets, 
including forcing posets for adding a generic function on $\omega_2$, 
adding a nonreflecting stationary subset of $\omega_2$, adding a 
Kurepa tree on $\omega_1$, and in \cite{krueger2} adding a club to a fat 
stationary subset of $\omega_2$. 
The main result of the present paper is to define a forcing poset using 
adequate sets which adds a $\Box_{\omega_1}$-sequence.

The idea of using models as side conditions in forcing goes back to 
{Todor\v cevi\' c} \cite{todor1}, where the method was applied to add generic 
objects of size $\omega_1$ with finite approximations. 
In the original context of applications of PFA, 
the preservation of $\omega_2$ was not necessary. 
To preserve $\omega_2$, {Todor\v cevi\' c} introduced the 
requirement of a system of isomorphisms on the models in a condition.

In the present paper we introduce the idea of a coherent adequate set of models. 
A coherent adequate set is essentially an adequate set in the sense of \cite{krueger1} 
which also satisfies the existence of a system of isomorphisms in the sense 
of {Todor\v cevi\' c}. 
Combining these two ideas turns out to provide a powerful method for 
forcing with side conditions. 
As an application we define a forcing poset which adds a square sequence on $\omega_2$ with finite conditions.

\bigskip

We assume that the reader is familiar with the basic theory of adequate sets 
as described in Sections 1--3 of \cite{krueger1}. 
Our treatment of coherent adequate sets owes a lot to the presentation of 
nicely arranged families given by Abraham and Cummings \cite{cummings}. 
Forcing a square sequence with finite conditions was first achieved 
by Dolinar and Dzamonja \cite{mirna} 
using the Mitchell style of models as side conditions \cite{mitchell}. 
An important difference is that 
the clubs which appear in the square sequence added by their forcing poset 
belong to the ground model, whereas for us the clubs are themselves 
generically approximated by finite fragments.

\section{Adequate Sets}

In this section we review the material on adequate sets which we will use. 
Throughout the paper we assume that 
$2^\omega = \omega_1$ and $2^{\omega_1} = \omega_2$.

Let $\pi$ be a bijection of $\omega_2$ onto $H(\omega_2)$. 
Fix a set of definable Skolem functions for the structure $(H(\omega_2),\in,\pi)$. 
For any set $a \subseteq \omega_2$, let $Sk(a)$ denote the closure of $a$ under 
these Skolem functions. 
Let $C^*$ be the club set of $\beta < \omega_2$ 
such that $Sk(\beta) \cap \omega_2 = \beta$. 
Let $\Lambda := C^* \cap \cof(\omega_1)$. 
Note that any ordinal in $\Lambda$ is also a limit point of $C^*$.

Let $\mathcal X$ be the set of countable $M \subseteq \omega_2$ such that 
$Sk(M) \cap \omega_2 = M$ and for all $\gamma \in M$, 
$\sup(C^* \cap \gamma) \in M$. 
Note that $\mathcal X$ is a club subset of $P_{\omega_1}(\omega_2)$. 
If $M \in \mathcal X$ then $Sk(M) = \pi[M]$. 
It follows that if $M$ and $N$ are in $\mathcal X$ and 
$N \in Sk(M)$, then $Sk(N) \in Sk(M)$. 
If $a$ and $b$ are in $\mathcal X \cup \Lambda$, then 
$Sk(a) \cap Sk(b) = Sk(a \cap b)$ (see Lemma 1.4 of \cite{krueger1}). 
This implies that if $M \in \mathcal X$ and $\beta \in \Lambda$, 
then $M \cap \beta \in \mathcal X$.

If $M \in \mathcal X$ and $\beta \in \Lambda \cap \sup(M)$, 
then $\min(M \setminus \beta)$ is in $\Lambda$. 
Clearly $\min(M \setminus \beta)$ has cofinality $\omega_1$. 
If this ordinal is not in $\Lambda$, then it is not a limit point of $C^*$. 
Also $\beta \ne \min(M \setminus \beta)$, so 
$\sup(M \cap \beta) < \beta < \min(M \setminus \beta)$. 
Hence $\sup(C^* \cap \min(M \setminus \beta))$ is below 
$\min(M \setminus \beta)$ and is in $M$ by the definition of $\mathcal X$. 
In particular this supremum is below $\beta$. 
This is a contradiction since $\beta$ is in $C^*$.

Let $M$ be in $\mathcal X$. 
A set $K$ is an \emph{initial segment} of $M$ if either $K = M$ or 
there exists $\beta \in M \cap \Lambda$ such that 
$K = M \cap \beta$. 
So any initial segment of $M$ is also in $\mathcal X$. 
If $M$ and $N$ are in $\mathcal X$ and $N \in Sk(M)$, 
then since $N$ has only countably many initial segments, 
they are all members of $Sk(M)$.

Since $2^{\omega} = \omega_1$, for all $\beta \in \Lambda$, 
$\mathcal X \cap P(\beta) \subseteq Sk(\beta)$. 
For since $\cf(\beta) = \omega_1$, 
every member of $\mathcal X \cap P(\beta)$ belongs to 
$P_{\omega_1}(\gamma)$ for some $\gamma < \beta$. 
And since $\omega_1 \subseteq Sk(\beta)$, 
$P_{\omega_1}(\gamma) \subseteq Sk(\beta)$. 
In particular, if $M \in \mathcal X$ and $\beta \in \Lambda$ 
then $M \cap \beta \in Sk(\beta)$.

For a set $M \in \mathcal X$, let $\Lambda_M$ denote the set of 
$\beta \in \Lambda$ such that 
$$
\beta = \min(\Lambda \setminus \sup(M \cap \beta)).
$$
In other words, $\beta \in \Lambda_M$ if $\beta \in \Lambda$ and 
there are no elements of 
$\Lambda$ strictly between $\sup(M \cap \beta)$ and $\beta$. 
For $M$ and $N$ in $\mathcal X$, $\Lambda_M \cap \Lambda_N$ 
has a largest element (see Lemma 2.4 of \cite{krueger1}). 
We denote this largest element by $\beta_{M,N}$, which is called the 
\emph{comparison point} of $M$ and $N$.
An important property of the comparison point is the following:
$$
(M \cup \lim(M)) \cap (N \cup \lim(N)) \subseteq \beta_{M,N}
$$
(see Proposition 2.6 of \cite{krueger1}).

\begin{definition}
A set $A \subseteq \mathcal X$ is \emph{adequate} if for all $M$ and $N$ 
in $A$, either $M \cap \beta_{M,N} \in Sk(N)$, 
$N \cap \beta_{M,N} \in Sk(M)$, or 
$M \cap \beta_{M,N} = N \cap \beta_{M,N}$. 
\end{definition}

Note that a set $A \subseteq \mathcal X$ 
is adequate iff for all $M$ and $N$ in $A$, 
$\{ M, N \}$ is adequate. 
If $\{ M, N \}$ is adequate, then $M \cap \beta_{M,N} \in Sk(N)$ iff 
$M \cap \omega_1 < N \cap \omega_1$, and 
$M \cap \beta_{M,N} = N \cap \beta_{M,N}$ 
iff $M \cap \omega_1 = N \cap \omega_1$.

Suppose that $\{ M, N \}$ is adequate. 
If $M \cap \beta_{M,N} = N \cap \beta_{M,N}$, then 
$M \cap N = M \cap \beta_{M,N}$. 
And if $M \cap \beta_{M,N} \in Sk(N)$, then 
$M \cap N = M \cap \beta_{M,N}$.

The next lemma records some important technical facts about 
comparison points which are used frequently.  
The proofs of these facts can be found in Section 3 of \cite{krueger1}.

\begin{lemma}
The following statements hold:
\begin{enumerate}
\item Let $M \in \mathcal X$, $\beta \in \Lambda$, 
and suppose $M \subseteq \beta$. 
Then for all $N \in \mathcal X$, $\beta_{M,N} \le \beta$.
\item Let $K, M, N \in \mathcal X$, and suppose $M \subseteq N$. 
Then $\beta_{K,M} \le \beta_{K,N}$.
\item Let $M$ and $N$ be in $\mathcal X$ and $\beta \in \Lambda$. 
If $\beta_{M,N} \le \beta$, then 
$\beta_{M,N} = \beta_{M \cap \beta,N}$.
\item Let $M$ and $N$ be in $\mathcal X$ and $\beta \in \Lambda$. 
If $N \subseteq \beta$, then 
$\beta_{M,N} = \beta_{M \cap \beta,N}$.
\end{enumerate}
\end{lemma}

Another important fact is that if $\{ M, N \}$ is adequate and $\beta \in \Lambda$, 
then $\{ M \cap \beta, N \cap \beta \}$ is adequate 
(see Lemma 3.3 of \cite{krueger1}).

\begin{lemma}
If $\{ M \cap \beta_{M,N}, N \cap \beta_{M,N} \}$ is adequate, then so is 
$\{ M, N \}$.
\end{lemma}

\begin{proof}
Let $\beta := \beta_{M,N}$. 
Since $\beta \le \beta$, Lemma 1.2(3) implies that 
$\beta = \beta_{M \cap \beta,N}$. 
And as $M \cap \beta \subseteq \beta$, Lemma 1.2(4) implies 
that $\beta_{M \cap \beta,N} = 
\beta_{M \cap \beta,N \cap \beta}$. 
Hence $\beta = 
\beta_{M \cap \beta,N \cap \beta}$. 
In particular, 
$(M \cap \beta) \cap \beta_{M \cap \beta,N \cap \beta} 
= M \cap \beta$ and 
$(N \cap \beta) \cap \beta_{M \cap \beta,N \cap \beta} = 
N \cap \beta$. 
So if $(M \cap \beta) \cap \beta_{M \cap \beta,N \cap \beta} \in 
Sk(N \cap \beta)$, then 
$M \cap \beta \in Sk(N)$, and similarly 
if $(N \cap \beta) \cap \beta_{M \cap \beta,N \cap \beta} \in 
Sk(M \cap \beta)$ then 
$N \cap \beta \in Sk(M)$. 
Also the equality 
$(M \cap \beta) \cap \beta_{M \cap \beta,N \cap \beta} = 
(N \cap \beta) \cap \beta_{M \cap \beta,N \cap \beta}$ is 
equivalent to the equality 
$M \cap \beta = N \cap \beta$.
\end{proof}

\section{Coherent Adequate Sets}

In the basic theory of adequate sets, we identify a set $M$ in $\mathcal X$ 
with $Sk(M)$, and oftentimes with the structure 
$(Sk(M),\in,\pi \cap Sk(M))$, which is an elementary substructure 
of $(H(\omega_2),\in,\pi)$. 
For any set $P \subseteq H(\omega_2)$ and $M \in \mathcal X$, 
let $P_M := P \cap Sk(M)$. 
In the context of coherent adequate sets we are interested in the expanded structure 
$$
\mathfrak M = (Sk(M),\in,\pi_M,{\mathcal X}_M,\Lambda_M).
$$
Note that $\mathfrak M$ is not necessarily an elementary substructure of 
$(H(\omega_2),\in,\pi,\mathcal X,\Lambda)$. 
In general if a set in $\mathcal X$ is denoted with a particular letter, 
we use the Fractur version of the letter to denote the above structure on 
its Skolem hull.

Let $M$ and $N$ be in $\mathcal X$. 
We say that $M$ and $N$ are \emph{isomorphic} if the structures 
$\mathfrak M$ and $\mathfrak N$ are isomorphic. 
In other words, $M$ and $N$ are isomorphic if there exists a bijection 
$\sigma : Sk(M) \to Sk(N)$ such that for all $x$ and $y$ in $Sk(M)$:
\begin{enumerate}
\item $x \in y$ iff $\sigma(x) \in \sigma(y)$;
\item $\pi(x) = y$ iff $\pi(\sigma(x)) = \sigma(y)$;
\item $x \in \mathcal X$ iff $\sigma(x) \in \mathcal X$;
\item $x \in \Lambda$ iff $\sigma(x) \in \Lambda$.
\end{enumerate}
In particular, such a map $\sigma$ is an isomorphism from 
$(Sk(M),\in)$ to $(Sk(N),\in)$. 
Since these structures model the axiom of extensionality, such an isomorphism 
is unique if it exists. 
In that case, 
let $\sigma_{M,N}$ denote the unique isomorphism from $\mathfrak M$ to 
$\mathfrak N$.
Note that if $M$, $N$, and $K$ are isomorphic, 
then $\sigma_{M,N} = \sigma_{K,N} \circ \sigma_{M,K}$.

For $M \in \mathcal X$, let $\overline{\mathfrak M}$ denote the transitive 
collapse of the structure $\mathfrak M$, and let 
$\sigma_M : \mathfrak M \to \overline{\mathfrak M}$ be the 
collapsing map. 
Note that $M$ and $N$ are isomorphic iff 
$\overline{\mathfrak M} = \overline{\mathfrak N}$. 
In that case, by the uniqueness of isomorphisms we have that 
$$
\sigma_{M,N} = \sigma_N^{-1} \circ \sigma_M.
$$

Suppose that $M$ and $N$ are isomorphic and $a \in Sk(M)$ is countable. 
We claim that $\sigma_{M,N}(a) = \sigma_{M,N}[a]$. 
Since $a$ and $\sigma_{M,N}(a)$ are countable, $a \subseteq Sk(M)$ 
and $\sigma_{M,N}(a) \subseteq Sk(N)$. 
Hence $x \in a$ implies $\sigma_{M,N}(x) \in \sigma_{M,N}(a)$, so that 
$\sigma_{M,N}[a] \subseteq \sigma_{M,N}(a)$. 
On the other hand, if $z \in \sigma_{M,N}(a)$, then for some $x \in Sk(M)$, 
$\sigma_{M,N}(x) = z$, which implies that $x \in a$. 
So $z \in \sigma_{M,N}[a]$.

\begin{lemma}
Let $M$ and $N$ be isomorphic and $K \in Sk(M) \cap \mathcal X$. 
Let $K^* = \sigma_{M,N}(K)$. 
Then $\sigma_{M,N}(Sk(K)) = Sk(K^*)$, $K$ and $K^*$ are isomorphic, 
and $\sigma_{M,N} \restriction Sk(K) = \sigma_{K,K^*}$.
\end{lemma}

\begin{proof}
Since $K$ is countable, $K^* = \sigma_{M,N}[K]$. 
For all $\alpha \in K$, we have that 
$\sigma_{M,N}(\pi(\alpha)) = \pi(\sigma_{M,N}(\alpha))$. 
It follows that 
$$
\sigma_{M,N}(Sk(K)) = 
\sigma_{M,N}[Sk(K)] = 
\sigma_{M,N}[\pi[K]] = 
\pi[\sigma_{M,N}[K]] = \pi[K^*] = Sk(K^*).
$$
So $\sigma_{M,N} \restriction Sk(K)$ is a bijection from 
$Sk(K)$ to $Sk(K^*)$, and it clearly preserves the predicates 
$\in$, $\pi$, $\mathcal X$, and $\Lambda$. 
Hence $\sigma_{M,N} \restriction Sk(K)$ is an isomorphism of 
$\mathfrak K$ to $\mathfrak K^*$. 
So $K$ and $K^*$ are isomorphic and $\sigma_{K,K^*} = 
\sigma_{M,N} \restriction Sk(K)$.
\end{proof}

\begin{lemma}
Let $M$ and $N$ be isomorphic, and let $K$ be an initial segment of $M$. 
Let $K^* := \sigma_{M,N}[K]$. 
Then $K^*$ is an initial segment of $N$, 
$\sigma_{M,N}[Sk(K)] = Sk(K^*)$, 
$K$ and $K^*$ are isomorphic, 
and $\sigma_{M,N} \restriction Sk(K) = \sigma_{K,K^*}$.
\end{lemma}

\begin{proof}
This is clear if $M = K$. 
Otherwise there is $\beta \in M \cap \Lambda$ such that 
$K = M \cap \beta$. 
Then $\sigma_{M,N}(\beta) \in N \cap \Lambda$, and easily 
$K^* = N \cap \sigma_{M,N}(\beta)$. 
By the argument from the previous lemma, 
$\sigma_{M,N}[Sk(K)] = Sk(\sigma_{M,N}[K]) = Sk(K^*)$, and 
$\sigma_{M,N} \restriction Sk(K)$ is an isomorphism of $Sk(K)$ 
to $Sk(K^*)$. 
Hence $K$ and $K^*$ are isomorphic and 
$\sigma_{M,N} \restriction Sk(K) = \sigma_{K,K^*}$.
\end{proof}

Suppose that $M \cap \beta_{M,N} = N \cap \beta_{M,N}$ and $M$ and 
$N$ are isomorphic.  
Applying the previous lemma, 
$\sigma_{M,N} \restriction (M \cap \beta_{M,N})$ 
is an isomorphism of $M \cap \beta_{M,N}$ to the initial 
segment $\sigma_{M,N}[M \cap \beta_{M,N}]$ of $N$. 
But the latter initial segment has the same order type as the initial segment 
$N \cap \beta_{M,N}$, 
so it is equal to it. 
Hence $\sigma_{M,N} \restriction Sk(M \cap \beta_{M,N})$ 
is an isomorphism of $Sk(M \cap \beta_{M,N})$ to itself, and therefore it is 
the identity map. 
But $M \cap \beta_{M,N} = M \cap N$. 
In particular, we have proven the following lemma.

\begin{lemma}
Let $\{ M, N \}$ be adequate, where $M$ and $N$ are isomorphic and 
$M \cap \beta_{M,N} = N \cap \beta_{M,N}$. 
Then $\sigma_{M,N} \restriction Sk(M \cap N)$ is the identity function.
\end{lemma}

We now introduce the most important idea of the paper.

\begin{definition}
Let $A \subseteq \mathcal X$. 
Then $A$ is a \emph{coherent adequate set} if $A$ is adequate and for all 
$M$ and $N$ in $A$:
\begin{enumerate}
\item if $M \cap \beta_{M,N} = 
N \cap \beta_{M,N}$, then $M$ and $N$ are isomorphic;
\item if $M \cap \beta_{M,N} \in Sk(N)$, then there exists $N'$ in $A$ 
such that $M \in Sk(N')$ and $N$ and $N'$ are isomorphic;
\item if $M \cap \beta_{M,N} = N \cap \beta_{M,N}$ and 
$K \in A \cap Sk(M)$, then $\sigma_{M,N}(K) \in A$.
\end{enumerate}
\end{definition}

Recall that if $A$ is adequate and $M$ and $N$ are in $A$, then 
$M \cap \beta_{M,N} \in Sk(N)$ iff $M \cap \omega_1 < N \cap \omega_1$, 
and $M \cap \beta_{M,N} = N \cap \beta_{M,N}$ iff 
$M \cap \omega_1 = N \cap \omega_1$. 
It follows that a finite adequate set $A$ is coherent iff the set 
$\{ Sk(M) : M \in A \}$ is a nicely arranged family in the sense of 
Definition 3.3 of \cite{cummings}. 

Also note that if $M$ and $N$ are in $\mathcal X$ and are isomorphic, then 
$M \cap \omega_1 = N \cap \omega_1$. 
For in that case $\sigma_{M,N}(\omega_1) = \omega_1$, and therefore 
$\sigma_{M,N}[M \cap \omega_1] = N \cap \omega_1$. 
But this implies that 
$M \cap \omega_1$ and $N \cap \omega_1$ have the same order type 
and thus are the same ordinal. 
Consequently the following are equivalent for $M$ and $N$ in a coherent 
adequate set: (1) $M \cap \omega_1 = N \cap \omega_1$; 
(2) $M \cap \beta_{M,N} = N \cap \beta_{M,N}$; 
(3) $M$ and $N$ are isomorphic.

\begin{lemma}
Let $A$ be a coherent adequate set. 
Let $M$ and $K$ be in $A$. 
If $K \cap \beta_{K,M} \in Sk(M)$, then 
there is $K^*$ in $A \cap Sk(M)$ such that 
$K$ and $K^*$ are isomorphic and 
$K \cap \beta_{K,M} = K^* \cap \beta_{K,M}$. 
\end{lemma}

\begin{proof}
Since $A$ is coherent, there exists $M'$ in $A$ such that 
$K \in Sk(M')$ and $M$ and $M'$ are isomorphic. 
Let $K^* = \sigma_{M',M}(K)$. 
Since $A$ is coherent, $K^* \in A$. 
By Lemma 2.1, $\sigma_{M',M} \restriction Sk(K)$ is an isomorphism of 
$Sk(K)$ to $Sk(K^*)$ and is equal to $\sigma_{K,K^*}$. 
And $\sigma_{M',M}$ is the identity on $M' \cap M = M' \cap \beta_{M,M'} = 
M \cap \beta_{M,M'}$. 
Since $K \subseteq M'$, $\beta_{K,M} \le \beta_{M',M}$.

Since $\sigma_{M',M} \restriction M' \cap \beta_{M',M}$ is the identity, 
$\sigma_{M',M}(K \cap \beta_{K,M}) = 
\sigma_{M',M}[K \cap \beta_{K,M}] = K \cap \beta_{K,M}$. 
Since $\sigma_{M',M} \restriction Sk(K) = \sigma_{K,K^*}$, 
Lemma 2.2 implies that $K \cap \beta_{K,M}$ is an initial segment of $K^*$. 
If $\gamma$ is in $K^* \setminus K$ and $\gamma < \beta_{K,M}$, 
then $\gamma < \beta_{M',M}$ implies that 
$\gamma = \sigma_{M,M'}(\gamma) \in K$ which is a contradiction. 
So $K \cap \beta_{K,M} = K^* \cap \beta_{K,M}$. 
\end{proof}

\begin{lemma}
Suppose that $A$ is a finite coherent adequate set, $N \in \mathcal X$, and 
$A \in Sk(N)$. 
Then $A \cup \{ N \}$ is a coherent adequate set. 
\end{lemma}

\begin{proof}
If $M \in A$ then since $M \in Sk(N)$, 
$M \cap \beta_{M,N} = M$, which is in $Sk(N)$. 
So $A \cup \{ N \}$ is adequate, and the requirements of being coherent 
are trivially satisfied.
\end{proof}

\begin{lemma}
Let $A$ be a coherent adequate set and $N \in A$. 
Then $A \cap Sk(N)$ is a coherent adequate set.
\end{lemma}

\begin{proof}
Clearly $A \cap Sk(N)$ is adequate, and (1) of Definition 2.4 is obvious. 
(3) is also straightforward. 
For (2), let $M$ and $K$ be in $A \cap Sk(N)$ and 
suppose that $K \cap \beta_{K,M} \in Sk(M)$. 
Since $A$ is coherent, there exists $M'$ in $A$ such that 
$K \in Sk(M')$ and $M$ and $M'$ are isomorphic. 
As $M \in Sk(N)$, $M' \cap \omega_1 = M \cap \omega_1 < N \cap \omega_1$. 
Hence $M' \cap \beta_{M',N} \in Sk(N)$. 
By Lemma 2.5 there exists $M^*$ in $A \cap Sk(N)$ such that 
$M'$ and $M^*$ are isomorphic and 
$M^* \cap \beta_{M',N} = M \cap \beta_{M',N}$. 
Now $K \in Sk(M') \cap Sk(N) = Sk(M' \cap N) = 
Sk(M' \cap \beta_{M',N}) = Sk(M^* \cap \beta_{M',N})$. 
So $K \in Sk(M^*)$, $M^* \in A \cap Sk(N)$, and 
$M^*$ and $M$ are isomorphic.
\end{proof}

\begin{lemma}
Let $A$ be a coherent adequate set. 
Suppose that $N$, $N'$, and $N^*$ are in $A$ and are isomorphic, 
where $N' \ne N^*$. 
Then $\sigma_{N',N} \restriction Sk(N' \cap N^*) = 
\sigma_{N^*,N} \restriction (N' \cap N^*)$, 
and for some $\beta \in N \cap \Lambda$, 
this function is an isomorphism 
of $Sk(N' \cap N^*)$ to $Sk(N \cap \beta)$. 
Also $\sigma_{N,N'} \restriction Sk(N \cap \beta) = 
\sigma_{N,N^*} \restriction Sk(N \cap \beta)$.
\end{lemma}

\begin{proof}
By Lemma 2.3, 
$\sigma_{N',N^*} \restriction Sk(N' \cap N^*)$ is the identity function. 
Also $\sigma_{N',N} = \sigma_{N^*,N} \circ \sigma_{N',N^*}$. 
So for any $x \in Sk(N' \cap N^*)$, 
$\sigma_{N',N}(x) = \sigma_{N^*,N}(\sigma_{N',N^*}(x)) = 
\sigma_{N^*,N}(x)$. 
This proves that 
$\sigma_{N',N} \restriction Sk(N' \cap N^*) = 
\sigma_{N^*,N} \restriction (N' \cap N^*)$. 
Denote this map by $\sigma$.
 
Since $N' \ne N^*$, $N' \cap N^*$ is a proper initial segment of $N'$ and of $N^*$. 
By Lemma 2.2, $\sigma[N' \cap N^*]$ is equal to 
$N \cap \beta$ for some $\beta \in N \cap \Lambda$, and 
$\sigma$ is an isomorphism of 
$Sk(N' \cap N^*)$ to $Sk(N \cap \beta)$. 
The last statement of the lemma 
follows from the fact that $\sigma_{N,N'} \restriction Sk(N \cap \beta)$ and 
$\sigma_{N,N^*} \restriction Sk(N \cap \beta)$ are both the inverse of $\sigma$.
\end{proof}

\section{Amalgamating Coherent Adequate Sets}

One of the main methods for preserving cardinals when forcing with 
models as side conditions is amalgamating conditions over elementary 
substructures. 
Proposition 3.5, which handles amalgamation over countable substructures, 
will be used to prove that the forcing poset in the 
next section is strongly proper and 
hence preserves $\omega_1$. 
Proposition 3.6 covers amalgamation over models of size $\omega_1$ 
and will be used to prove that the forcing poset in the next section is 
$\omega_2$-c.c.

The next four technical lemmas will be used to prove 
Proposition 3.5.

\begin{lemma}
Let $M$ and $N$ be in $\mathcal X$ and suppose that 
$M$ and $N$ are isomorphic. 
If $\alpha < \gamma$ are in $M$ and 
$\Lambda \cap [\alpha,\gamma] = \emptyset$, 
then $\Lambda \cap [\sigma_{M,N}(\alpha),\sigma_{M,N}(\gamma)] 
= \emptyset$. 
\end{lemma}

\begin{proof}
Suppose for a contradiction that $\zeta$ is in 
$\Lambda \cap [\sigma_{M,N}(\alpha),\sigma_{M,N}(\gamma)]$. 
Let $\zeta^* = \min(N \setminus \zeta)$. 
Then $\zeta^* \in N \cap \Lambda \cap 
[\sigma_{M,N}(\alpha),\sigma_{M,N}(\gamma)]$. 
Therefore $\sigma_{N,M}(\zeta^*) \in \Lambda \cap [\alpha,\gamma]$, 
which contradicts that $\Lambda \cap [\alpha,\gamma] = \emptyset$.
\end{proof}

\begin{lemma}
Let $M$ and $N$ be in $\mathcal X$. 
Let $\alpha \le \gamma$ be ordinals, where $\alpha \in M \cup \lim(M)$ and 
$\gamma \in N \cup \lim(N)$. 
If $\Lambda \cap [\alpha,\gamma] = \emptyset$, then 
$\gamma < \beta_{M,N}$.
\end{lemma}

\begin{proof}
Let $\beta = \min(\Lambda \setminus \gamma)$. 
Then $\gamma \le \sup(N \cap \beta)$, so 
$\beta = \min(\Lambda \setminus \sup(N \cap \beta))$. 
Also $\alpha \le \sup(M \cap \beta)$, and since 
$\Lambda \cap [\alpha,\gamma] = \emptyset$, 
$\beta = \min(\Lambda \setminus \sup(M \cap \beta))$. 
Therefore $\beta \in \Lambda_M \cap \Lambda_N$, which implies that 
$\beta \le \beta_{M,N}$. 
Since $\gamma$ is not in $\Lambda$, it follows that $\gamma < \beta_{M,N}$.
\end{proof}

\begin{lemma}
Let $M$, $N$, $K$, and $P$ be in $\mathcal X$, where $M$ and $N$ are 
isomorphic and $K$ and $P$ are in $Sk(M)$. 
Let $\sigma := \sigma_{M,N}$, 
$K^* := \sigma(K)$, and $P^* = \sigma(P)$. 
Suppose that $\beta = \min(M \setminus \beta_{K,P})$. 
Then $\sigma(\beta) = 
\min(N \setminus \beta_{K^*,P^*})$.
\end{lemma}

\begin{proof}
Let $\alpha = \sup(K \cap \beta)$ and $\gamma = \sup(P \cap \beta)$. 
Without loss of generality assume that $\alpha \le \gamma$. 
Since $\alpha$ and $\gamma$ have cofinality $\omega$, they are 
not in $\Lambda$. 
And as $\alpha$ and $\gamma$ are in $M$ and below $\beta$, 
$\alpha$ and $\gamma$ are less than $\beta_{K,P}$. 
Thus $\alpha = \sup(K \cap \beta_{K,P})$ and 
$\gamma = \sup(P \cap \beta_{K,P})$. 

Since $\beta_{K,P} \in \Lambda_K \cap \Lambda_P$, 
$\beta_{K,P} = \min(\Lambda \setminus \alpha) = 
\min(\Lambda \setminus \gamma)$. 
So $\Lambda \cap [\alpha,\gamma] = \emptyset$. 
By Lemma 3.1 it follows that 
$\Lambda \cap [\sigma(\alpha),\sigma(\gamma)] = \emptyset$. 
Since $\sigma(\alpha) \in \lim(K^*)$ and $\sigma(\gamma) \in \lim(P^*)$, 
Lemma 3.2 implies that $\beta_{K^*,P^*} > \sigma(\gamma)$. 

By the definition of $\beta$, $\sup(M \cap \beta) < \beta_{K,P}$. 
Since $\beta_{K,P} = \min(\Lambda \setminus \gamma)$, 
it follows that for all $\gamma' \in M \cap [\gamma,\beta)$, 
$\Lambda \cap [\gamma,\gamma'] = \emptyset$. 
Hence by Lemma 3.1, for all $\gamma^* \in N \cap [\sigma(\gamma),\sigma(\beta))$, 
$\Lambda \cap [\sigma(\gamma),\gamma^*] = \emptyset$. 
Therefore $\beta_{K^*,P^*} > \sup(N \cap \sigma(\beta))$.

We will be done if we can show that $\beta_{K^*,P^*} \le \sigma(\beta)$. 
Suppose for a contradiction that $\beta_{K^*,P^*} > \sigma(\beta)$. 
Let $\tau = \sup(K^* \cap \beta_{K^*,P^*})$ and 
$\xi = \sup(P^* \cap \beta_{K^*,P^*})$. 
Without loss of generality assume that $\tau \le \xi$, since the other case 
follows by a symmetric argument. 
So $\beta_{K^*,P^*} = \min(\Lambda \setminus \tau) = 
\min(\Lambda \setminus \xi)$. 
Since $\beta_{K^*,P^*} > \sigma(\beta)$ and $\sigma(\beta) \in \Lambda$, 
$\tau$ and $\xi$ 
are greater than $\sigma(\beta)$. 
Also clearly $\Lambda \cap [\tau,\xi] = \emptyset$. 
By Lemma 3.1, $\Lambda \cap [\sigma^{-1}(\tau),\sigma^{-1}(\xi)] = \emptyset$. 
Since $\sigma^{-1}(\tau) \in \lim(K)$ and 
$\sigma^{-1}(\xi) \in \lim(P)$, 
Lemma 3.2 implies that $\beta_{K,P} > \sigma^{-1}(\xi)$. 
But $\xi > \sigma(\beta)$ implies that $\sigma^{-1}(\xi) > \beta$. 
Hence $\beta_{K,P} > \beta$, which is a contradiction.
\end{proof}

\begin{lemma}
Let $M$, $N$, $K$, and $P$ be in $\mathcal X$. 
Suppose that $M$ and $N$ are isomorphic and $K$ and $P$ are in $Sk(M)$. 
If $\{ K, P \}$ is adequate, 
then $\{ \sigma_{M,N}(K), \sigma_{M,N}(P) \}$ is adequate.
\end{lemma}

\begin{proof}
Let $\sigma := \sigma_{M,N}$, $K^* := \sigma_{M,N}(K)$, and 
$P^* := \sigma_{M,N}(P)$. 
By symmetry it suffices to consider the cases when 
$K \cap \beta_{K,P} \in Sk(P)$ and $K \cap \beta_{K,P} = P \cap \beta_{K,P}$. 
First assume that $\beta_{K,P} \ge \sup(M)$. 
Then $K \cap \beta_{K,P} = K$ and $P \cap \beta_{K,P} = P$. 
If $K \cap \beta_{K,P} \in Sk(P)$, then $K \in Sk(P)$. 
So $\sigma(K) \in \sigma(Sk(P)) = Sk(\sigma(P))$. 
Also if $K \cap \beta_{K,P} = P \cap \beta_{K,P}$, then $K = P$, which implies 
that $\sigma(K) = \sigma(P)$.

Now assume that $\beta_{K,P} < \sup(M)$. 
Let $\beta := \min(M \setminus \beta_{K,P})$. 
Then $K \cap \beta = K \cap \beta_{K,P}$ and 
$P \cap \beta = P \cap \beta_{K,P}$. 
By Lemma 3.3, $\sigma(\beta) = 
\min(N \setminus \beta_{K^*,P^*})$. 
Therefore $K^* \cap \sigma(\beta) = 
K^* \cap \beta_{K^*,P^*}$ and 
$P^* \cap \sigma(\beta) = 
P^* \cap \beta_{K^*,P^*}$.

Suppose that $K \cap \beta_{K,P} \in Sk(P)$. 
Then $K \cap \beta \in Sk(P)$. 
So $\sigma(K \cap \beta) = K^* \cap \sigma(\beta) 
\in \sigma(Sk(P)) = Sk(P^*)$. 
Therefore $K^* \cap \beta_{K^*,P^*} \in Sk(P^*)$. 
Now suppose that 
$K \cap \beta_{K,P} = P \cap \beta_{K,P}$. 
Then $K \cap \beta = P \cap \beta$. 
So $K^* \cap \sigma(\beta) = \sigma(K \cap \beta) = 
\sigma(P \cap \beta) = P^* \cap \sigma(\beta)$. 
Hence $K^* \cap \beta_{K^*,P^*} = 
P^* \cap \beta_{K^*,P^*}$.
\end{proof}

The following proposition describes amalgamation of coherent adequate sets over 
countable elementary substructures. 
It will be used to prove that the forcing poset in the next section 
is strongly proper.

\begin{proposition}
Let $A$ be a coherent adequate set and $N \in A$. 
Suppose that $B$ is a coherent adequate set and 
$A \cap Sk(N) \subseteq B \subseteq Sk(N)$. 
Let $C$ be the set 
$$
\{ M \in A : N \cap \omega_1 \le M \cap \omega_1 \} \cup 
\{ \sigma_{N,N'}(K) : N' \in A, \ N \cap \omega_1 = N' \cap \omega_1, \ 
K \in B \}.
$$
Then $C$ is a coherent adequate set which contains $A \cup B$.
\end{proposition}

\begin{proof}
First we prove that $C$ is adequate. 
Obviously any two sets in $\{ M \in A : N \cap \omega_1 \le M \cap \omega_1 \}$ 
compare properly since $A$ is adequate. 
Consider $M \in A$ with $N \cap \omega_1 \le M \cap \omega_1$, and 
$L = \sigma_{N,N'}(K)$ for some $N' \in A$ with 
$N \cap \omega_1 = N' \cap \omega_1$ and 
some $K \in B$. 
Since $N' \cap \omega_1 = N \cap \omega_1 \le M \cap \omega_1$, 
the set $N' \cap \beta_{M,N'}$ is either in $Sk(M)$ 
or is equal to $M \cap \beta_{M,N'}$. 
In either case, $Sk(N' \cap \beta_{M,N'})$ is a subset of $Sk(M)$. 
Since $L \subseteq N'$, $\beta_{L,M} \le \beta_{M,N'}$. 
As $L$ is in $Sk(N')$, $L \cap \beta_{L,M}$ is in 
$Sk(N') \cap Sk(\beta_{M,N'}) = Sk(N' \cap \beta_{M,N'})$. 
Hence $L \cap \beta_{L,M}$ is a member of $Sk(M)$.

Now consider $M$ and $L$ such that $M = \sigma_{N,N'}(K)$ for some 
$N' \in A$ with $N \cap \omega_1 = N' \cap \omega_1$ and some $K \in B$, and 
$L = \sigma_{N,N^*}(P)$ for some $N^* \in A$ with 
$N \cap \omega_1 = N^* \cap \omega_1$ and some $P \in B$. 
Since $B$ is adequate, $K$ and $P$ compare properly. 
If $N' = N^*$, then $\{ M, L \}$ is adequate by Lemma 3.4. 
Suppose $N' \ne N^*$. 
By symmetry it suffices to consider the cases when 
$K \cap \beta_{K,P}$ is either in $Sk(P)$ or is equal to $P \cap \beta_{K,P}$.

The sets $N'$ and $N^*$ are isomorphic, and 
$N' \cap \beta_{N',N^*} = N^* \cap \beta_{N',N^*} = N' \cap N^*$. 
By Lemma 2.8, $\sigma_{N',N} \restriction N' \cap N^* = 
\sigma_{N^*,N} \restriction N' \cap N^*$, and there exists 
$\beta \in N \cap \Lambda$ such that 
$N \cap \beta = \sigma_{N',N}[N' \cap N^*]$. 
Let $\sigma := \sigma_{N,N'} \restriction Sk(N \cap \beta)$. 
By Lemma 2.8, $\sigma = \sigma_{N,N^*} \restriction Sk(N \cap \beta)$ and 
$\sigma$ is an isomorphism of $Sk(N \cap \beta)$ to $Sk(N' \cap N^*)$. 
Now $\sigma(K \cap \beta) = 
\sigma_{N,N'}[K \cap (N \cap \beta)] = 
\sigma_{N,N'}[K] \cap \sigma_{N,N'}[N \cap \beta] = 
M \cap (N' \cap \beta_{N',N^*}) = M \cap \beta_{N',N^*}$, and similarly 
$\sigma(P \cap \beta) = L \cap \beta_{N',N^*}$.

Since $\{ K, P \}$ is adequate, so is $\{ K \cap \beta, P \cap \beta \}$. 
By Lemma 3.4, it follows that 
$\{ \sigma(K \cap \beta), \sigma(P \cap \beta) \}$ is adequate. 
In other words, 
$\{ M \cap \beta_{N',N^*}, L \cap \beta_{N',N^*} \}$ is adequate. 
Since $M \subseteq N'$, $\beta_{L,M} \le \beta_{L,N'}$, and since 
$L \subseteq N^*$, $\beta_{L,N'} \le \beta_{N',N^*}$. 
Hence $\beta_{L,M} \le \beta_{N',N^*}$. 
Therefore $\{ M \cap \beta_{L,M}, L \cap \beta_{L,M} \}$ is adequate. 
By Lemma 1.3 it follows that $\{ M, L \}$ is adequate.

Now we show that $A \cup B \subseteq C$ and $C$ is coherent. 
This statement follows immediately from Lemmas 3.8 and 3.9 
of \cite{cummings}; we include a proof for completeness. 
If $K \in B$, then $K = \sigma_{N,N}(K)$ is in $C$ by definition. 
Let $M \in A$. 
If $N \cap \omega_1 \le M \cap \omega_1$, then $M \in C$ by definition. 
Otherwise $M \cap \omega_1 < N \cap \omega_1$. 
So there exists $N' \in A$ 
isomorphic to $N$ such that $M \in Sk(N')$. 
Let $K := \sigma_{N',N}(M)$, which is in $A \cap Sk(N)$ and hence in $B$. 
Then $M = \sigma_{N,N'}(K)$ is in $C$.

Suppose that $L$ and $M$ are in $C$ and 
$L \cap \omega_1 = M \cap \omega_1$. 
We will show that $L$ and $M$ are isomorphic. 
If $M \cap \omega_1 \ge N \cap \omega_1$, then $L$ and $M$ are in $A$ 
and hence are isomorphic. 
Otherwise $M = \sigma_{N,N'}(M^*)$ and $L = \sigma_{N,N''}(L^*)$, where 
$M^*$ and $L^*$ are in $B$ and $N'$ and $N''$ are in $A$ 
and are isomorphic to $N$. 
Then $M^* \cap \omega_1 = L^* \cap \omega_1$, which implies that 
$M^*$ and $L^*$ are isomorphic. 
It follows that $M$ and $L$ are isomorphic.

Assume that $L$ and $M$ are in $C$ and 
$L \cap \omega_1 < M \cap \omega_1$. 
We will show that there is $M'$ in $C$ isomorphic to $M$ such that 
$L \in Sk(M')$. 
If $N \cap \omega_1 \le L \cap \omega_1$, then $L$ and $M$ are in $A$ 
and we are done. 
Suppose that $L \cap \omega_1 < N \cap \omega_1 \le M \cap \omega_1$. 
Then $L = \sigma_{N,N'}(L^*)$ for some $L^*$ in $B$ and $N' \in A$ 
which is isomorphic to $N$. 
Fix $M'$ in $A$ which is isomorphic to $M$ such that $N'$ 
is either equal to $M'$ or is a member of $Sk(M')$. 
Then $L \in Sk(M')$ and we are done.

Assume that $M \cap \omega_1 < N \cap \omega_1$. 
Then $L = \sigma_{N,N'}(L^*)$ and $M = \sigma_{N,N''}(M^*)$, where 
$L^*$ and $M^*$ are in $B$ and $N'$ and $N''$ are in $A$ and are both 
isomorphic to $N$. 
Since $L^* \cap \omega_1 < M^* \cap \omega_1$, 
there is $M^{**}$ in $B$ isomorphic to $M^*$ such that 
$L^* \in Sk(M^{**})$. 
Then $\sigma_{N,N'}(M^{**})$ is in $C$, is 
isomorphic to $M^{**}$ and hence to 
$M$, and its Skolem hull contains $L$.

Now assume that $M$, $K$, and $L$ are in $C$, 
$M \cap \omega_1 = K \cap \omega_1$, and $L \in C \cap Sk(M)$. 
We will show that $\sigma_{M,K}(L) \in C$. 
First assume that $N \cap \omega_1 \le M \cap \omega_1$. 
Then $M$ and $K$ are in $A$. 
If $L \in A$ then we are done. 
So assume that $L = \sigma_{N,N'}(L^*)$ for some $L^* \in B$ and 
$N'$ in $A$ isomorphic to $N$. 
Fix $J$ in $A$ isomorphic to $M$ such that $N'$ is either equal to $J$ 
or a member of $Sk(J)$. 
Let $N'' := \sigma_{J,M}(N')$ and let 
$N''' := \sigma_{M,K}(N'')$. 
Then $N''$ and $N'''$ are in $A$. 
So $\sigma_{N,N'''}(L^*) \in C$. 
Since $L$ is in $Sk(J) \cap Sk(M)$, 
$\sigma_{J,M}(L) = L$. 
Then $\sigma_{N,N'''}(L^*) = 
\sigma_{N'',N'''}(\sigma_{N',N''}(\sigma_{N,N'}(L^*))) = 
\sigma_{N'',N'''}(\sigma_{N',N''}(L)) = 
\sigma_{M,K}(\sigma_{J,M}(L)) = \sigma_{M,K}(L)$. 
So $\sigma_{M,K}(L) \in C$.

Finally, assume that $M \cap \omega_1 < N \cap \omega_1$. 
Then $M = \sigma_{N,N'}(M^*)$, $K = \sigma_{N,N''}(K^*)$, and 
$L = \sigma_{N,N'''}(L^*)$, where $M^*$, $K^*$, and $L^*$ are in $B$, 
and $N'$, $N''$, and $N'''$ are in $A$ and are isomorphic to $N$. 
Since $L \in Sk(M)$, $L \in Sk(N') \cap Sk(N''')$. 
So $\sigma_{N',N}(L) = \sigma_{N''',N}(L) = L^*$. 
So $\sigma_{M,M^*}(L) = \sigma_{N',N}(L) = L^*$. 
Then $\sigma_{M,K}(L) = 
\sigma_{K^*,K}(\sigma_{M^*,K^*}(\sigma_{M,M^*}(L))) = 
\sigma_{K^*,K}(\sigma_{M^*,K^*}(L^*)) = 
\sigma_{N,N''}(\sigma_{M^*,K^*}(L^*))$. 
Since $L^* \in B$, $\sigma_{M^*,K^*}(L^*) \in B$. 
Hence $\sigma_{N,N''}(\sigma_{M^*,K^*}(L^*)) \in C$. 
So $\sigma_{M,K}(L) \in C$.
\end{proof}

The next result describes amalgamation of coherent adequate sets over models 
of size $\omega_1$. 
It will be used to show that the forcing poset in the next section 
is $\omega_2$-c.c.

\begin{proposition}
Let $A$ be a coherent adequate set and $\beta \in \Lambda$. 
Let $A^+ := \{ M \in A : M \setminus \beta \ne \emptyset \}$ and 
$A^{-} := \{ M \in A : M \subseteq \beta \}$. 
Suppose that $\beta^* \in \beta \cap \Lambda$ and 
for all $M \in A$, $\sup(M \cap \beta) < \beta^*$. 
Assume that there exists a map $M \mapsto M'$ from $A^+$ into 
$\mathcal X \cap Sk(\beta)$ satisfying that for all $M$ and $K$ in $A^+$:
\begin{enumerate}
\item $M$ and $M'$ are isomorphic and 
$M \cap \beta^* = M' \cap \beta^*$;
\item $K \in Sk(M)$ iff $K' \in Sk(M')$;
\item if $K \in Sk(M)$ then $\sigma_{M,M'}(K) = K'$;
\item $A^{-} \cup \{ M' : M \in A^+ \}$ is a coherent adequate set.
\end{enumerate}
Then $C := A \cup \{ M' : M \in A^+ \}$ is a coherent adequate set.
\end{proposition}

\begin{proof}
Note that by assumption (1), $\sigma_{M,M'} \restriction \beta^*$ is 
the identity function for all $M \in A^+$. 
Let us begin by proving that $C$ is adequate. 
Note that if $M \in A^+$, then $M$ and $M'$ have the same order type, 
which is larger than the order type of $M \cap \beta^* = M' \cap \beta^*$; 
it follows that $M' \setminus \beta^*$ is nonempty. 
Therefore $C$ is the union of the three disjoint sets 
$A^{-}$, $A^{+}$, and $\{ M' : M \in A^+ \}$. 
By (4) and the fact that $A$ is adequate, it suffices to compare 
a set in $A^+$ with a set in $\{ M' : M \in A^+ \}$.

Let $K$ and $M$ be in $A^+$, and let us compare $K$ and $M'$. 
Since $M' \subseteq \beta$, $\beta_{K,M'} \le \beta$ by Lemma 1.2(1). 
Hence $\beta_{K,M'} = \beta_{K \cap \beta,M'}$ by Lemma 1.2(3). 
But $K \cap \beta = K \cap \beta^*$, which implies by Lemma 1.2(1,4) that 
$\beta_{K,M'} = 
\beta_{K \cap \beta,M'} = 
\beta_{K \cap \beta^*,M' \cap \beta^*} \le \beta^*$. 
Also $K \cap \beta^* = K' \cap \beta^*$ and 
$M' \cap \beta^* = M \cap \beta$. 
Now $\beta_{K,M'} = \beta_{K \cap \beta^*,M' \cap \beta^*}$, 
and since $K \cap \beta^* \subseteq K$ and 
$M' \cap \beta^* \subseteq M$, it follows that 
$\beta_{K,M'} \le \beta_{K,M}$.

We split into cases depending on the comparison of $K$ and $M$. 
Suppose that $K \cap \beta_{K,M} \in Sk(M)$. 
Since $\beta_{K,M'} \le \beta^*, \beta_{K,M}$, 
it follows that $K \cap \beta_{K,M'} \in 
Sk(M) \cap Sk(\beta^*) = Sk(M \cap \beta^*) = 
Sk(M' \cap \beta^*)$. 
Therefore $K \cap \beta_{K,M'} \in Sk(M')$. 
Now assume that $M \cap \beta_{K,M} \in Sk(K)$. 
Since $\beta_{K,M'} \le \beta_{K,M}$, 
$M \cap \beta_{K,M'} \in Sk(K)$. 
But $\beta_{K,M'} \le \beta^*$ implies that 
$M \cap \beta_{K,M'} = M' \cap \beta_{K,M'}$. 
So $M' \cap \beta_{K,M'} \in Sk(K)$. 
Now assume that $K \cap \beta_{K,M} = M \cap \beta_{K,M}$. 
Since $\beta_{K,M'} \le \beta_{K,M}$, 
$K \cap \beta_{K,M'} = M \cap \beta_{K,M'}$. 
But $\beta_{K,M'} \le \beta^*$, so 
$M \cap \beta_{K,M'} = M' \cap \beta_{K,M'}$. 
Hence $K \cap \beta_{K,M'} = M' \cap \beta_{K,M'}$.

Now we show that $C$ is coherent. 
Recall that $A$ is the union of the three disjoint sets 
$A^{+}$, $A^{-}$, and $\{ M' : M \in A^+ \}$. 
The union of the first and second set is equal to $A$, which is coherent, and the 
union of the second and third set is coherent by (4). 
Note that requirements (1) and (2) in the definition of coherence 
follow immediately from these facts, 
except for the case of a pair of models where one 
is in $A^+$ and the other is in $\{ M' : M \in A^{+} \}$.

Let $K$ and $M$ be in $A^+$, and we verify requirements (1) and (2) 
for $K$ and $M'$. 
Suppose that $K \cap \beta_{K,M'} = M' \cap \beta_{K,M'}$. 
Then $K \cap \omega_1 = M \cap \omega_1$. 
Since $A$ is coherent, $K$ and $M$ are isomorphic. 
Hence $K$ and $M'$ are isomorphic.

Suppose that $K \cap \beta_{K,M'} \in Sk(M')$. 
Then $K \cap \omega_1 < M \cap \omega_1$, 
so $K \cap \beta_{K,M} \in Sk(M)$. 
So there exists $M^*$ in $A$ such that $K \in Sk(M^*)$ and 
$M$ and $M^*$ are isomorphic. 
Hence $M^*$ and $M'$ are isomorphic. 
Now assume that $M' \cap \beta_{K,M'} \in Sk(K)$. 
Then $M' \cap \omega_1 < K' \cap \omega_1$, so 
$M' \cap \beta_{K',M'} \in Sk(K')$. 
Since $A^{-} \cup \{ L' : L \in A^+ \}$ is coherent, 
there is $K^*$ in $C$ such that $M' \in Sk(K^*)$ and 
$K^*$ and $K'$ are isomorphic. 
Then $K^*$ and $K$ are isomorphic.

Now we prove that requirement (3) holds of $C$. 
Let $M_1$ and $M_2$ be in $C$ with 
$M_1 \cap \beta_{M_1,M_2} = M_2 \cap \beta_{M_1,M_2}$ 
and let $K \in C \cap Sk(M_1)$. 
We will prove that $\sigma_{M_1,M_2}(K)$ is in $C$. 
Note that if $M_1$ and $M_2$ are both in $A$, then so is $K$, and if 
$M_1$ and $M_2$ are both in $A^- \cup \{ M' : M \in A^+ \}$, 
then so is $K$. 
Since $A$ and $A^- \cup \{ M' : M \in A^+ \}$ are both coherent, 
we are done in these cases. 
So again it suffices to prove (3) in the case of two sets, where one 
is in $A^+$ and the other is in $\{ M' : M \in A^+ \}$. 

Assume that $M_1$ is in $A^+$ and $M_2 = M'$ for some 
$M \in A^+$. 
Then $M_1$ and $M$ are isomorphic. 
Since $K \in Sk(M_1)$, $K \cap \beta \subseteq \beta^*$, and hence 
$K$ is in $A$. 
As $A$ is coherent, $P := \sigma_{M_1,M}(K) \in A \cap Sk(M)$. 
If $P \in A^-$, then since $\sigma_{M,M'} \restriction \beta^*$ 
is the identity, $\sigma_{M,M'}(P) = P$. 
Hence $\sigma_{M_1,M'}(K) = \sigma_{M,M'}(P) = P$ is in $A$. 
Otherwise $P \in A^+$, and by assumption (3), 
$\sigma_{M,M'}(P) = P'$. 
So $\sigma_{M_1,M'}(K) = \sigma_{M,M'}(\sigma_{M_1,M}(K)) = 
\sigma_{M,M'}(P) = P' \in C$.

In the last case assume that $M_1 = M'$ for some $M \in A^+$ 
and $M_2 \in A^+$. 
Since $K \in Sk(M')$, $K \subseteq \beta$, so $K$ is not in $A^+$. 
Suppose that $K$ is in $A^-$. 
Then $K$ is a subset of $\beta^*$, so 
$\sigma_{M',M}(K) = K$. 
Hence $K$ is in $Sk(M) \cap A$, and therefore $\sigma_{M,M_2}(K) \in A$ 
since $A$ is coherent. 
But $\sigma_{M',M_2}(K) = \sigma_{M,M_2}(\sigma_{M',M}(K)) 
= \sigma_{M,M_2}(K) \in C$. 
Otherwise $K$ is equal to $P'$ for some $P \in A^+$. 
So $P' \in Sk(M')$. 
By assumptions (3) and (4), $P \in Sk(M)$ and $\sigma_{M,M'}(P) = P'$. 
Since $P$ is in $A$ and $A$ is coherent, $\sigma_{M,M_2}(P) \in A$. 
So $\sigma_{M',M_2}(K) = \sigma_{M',M_2}(P') = 
\sigma_{M',M_2}(\sigma_{M,M'}(P)) = 
\sigma_{M,M_2}(P) \in C$.
\end{proof}

\section{Forcing Square with Finite Conditions}

We define a forcing poset which adds a square sequence with finite conditions, 
using coherent adequate sets as side conditions.

By a \emph{triple} we mean a sequence 
$\langle \alpha, \gamma, \beta \rangle$, where $\alpha \in \Lambda$ and 
$\gamma < \beta < \alpha$. 
Given distinct triples $\langle \alpha, \gamma, \beta \rangle$ and 
$\langle \alpha', \gamma', \beta' \rangle$, we say that the triples are 
\emph{nonoverlapping} if either $\alpha \ne \alpha'$, 
or $\alpha = \alpha'$ and $[\gamma,\beta) \cap [\gamma',\beta') = \emptyset$; 
otherwise they are \emph{overlapping}. 
Given a triple $\langle \alpha, \gamma, \beta \rangle$ and $M \in \mathcal X$, 
we say that $\langle \alpha, \gamma, \beta \rangle$ and $M$ 
are \emph{nonoverlapping} 
if $\alpha \in M$ implies that either 
$\gamma$ and $\beta$ are in $M$ or $\sup(M \cap \alpha) < \gamma$; 
otherwise they are \emph{overlapping}.

Clearly if $M$ and $N$ are isomorphic and $a$ and $b$ are 
nonoverlapping triples in $Sk(M)$, then 
$\sigma_{M,N}(a)$ and $\sigma_{M,N}(b)$ are nonoverlapping triples. 
And if $K \in Sk(M) \cap \mathcal X$ and $a$ and $K$ are nonoverlapping, 
then $\sigma_{M,N}(a)$ and $\sigma_{M,N}(K)$ are nonoverlapping.

\begin{definition}
Let $\p$ be the forcing poset whose conditions are pairs 
$( x, A )$ satisfying:
\begin{enumerate}
\item $x$ is a finite pairwise nonoverlapping set of triples;
\item $A$ is a finite coherent adequate set;
\item for all $M \in A$ and $\langle \alpha, \gamma, \beta \rangle \in x$, 
$M$ and $\langle \alpha, \gamma, \beta \rangle$ are nonoverlapping;
\item if $M$ and $M'$ are in $A$ and 
$M \cap \beta_{M,M'} = M' \cap \beta_{M,M'}$, then for any 
triple $\langle \alpha, \gamma, \beta \rangle \in Sk(M) \cap x$, 
$\sigma_{M,M'}(\langle \alpha, \gamma, \beta \rangle) \in x$.
\end{enumerate}
Let $(y,B) \le (x,A)$ if $x \subseteq y$ and $A \subseteq B$.
\end{definition}

If $p = (x,A)$, we write $x_p := x$ and $A_p := A$.

We will prove that $\p$ preserves all cardinals. 
For each $\alpha \in \Lambda$, let $\dot c_\alpha$ be a $\p$-name for 
the set 
$$
\{ \gamma : \exists p \in \dot G \ \exists \beta \ 
( \langle \alpha, \gamma, \beta \rangle \in x_p ) \}.
$$
We will show that each $\dot c_\alpha$ is a cofinal subset of $\alpha$ 
with order type $\omega_1$, and whenever $\xi$ is a common 
limit point of $\dot c_{\alpha}$ and $\dot c_{\alpha'}$, 
$\dot c_\alpha \cap \xi = \dot c_{\alpha'} \cap \xi$.

\begin{lemma}
Let $A$ be a coherent adequate set and $x$ a set of triples. 
Let $y$ be the set 
$$
x \cup \{ \sigma_{M,M'}(a) : M, M' \in A, \ 
M \cap \omega_1 = M' \cap \omega_1, \ a \in x \cap Sk(M) \}.
$$
Then for all $N$ and $N'$ in $A$ which are isomorphic and any $a \in y$, 
$\sigma_{N,N'}(a) \in y$.
\end{lemma}

\begin{proof}
Let $N$ and $N'$ be isomorphic sets in $A$ and $a \in y$. 
If $a \in x$, then $\sigma_{N,N'}(a) \in y$ by definition. 
Otherwise there are $M$ and $M'$ in $A$ which are isomorphic and 
$b$ in $x$ such that $a = \sigma_{M,M'}(b)$. 
So $a$ is in $Sk(M') \cap Sk(N) = Sk(M' \cap N)$. 

First assume that $M' \cap \beta_{M',N} \in Sk(N)$. 
By Lemma 2.5 there is $M^*$ in $Sk(N)$ which is isomorphic to 
$M'$ such that $M' \cap \beta_{M',N} = M^* \cap \beta_{M',N}$. 
In particular, $a \in Sk(M' \cap N) = Sk(M' \cap \beta_{M',N}) = 
Sk(M^* \cap \beta_{M',M})$. 
By Lemma 2.8, $\sigma_{M^*,M}(a) = \sigma_{M',M}(a) = b$. 
So $\sigma_{M,M^*}(b) = \sigma_{M,M'}(b) = a$. 
Let $P := \sigma_{N,N'}(M^*)$. 
Then $\sigma_{N,N'} \restriction Sk(M^*) = \sigma_{M^*,P}$. 
By Lemma 2.8, $\sigma_{M',P} \restriction Sk(M' \cap M^*) = 
\sigma_{M^*,P} \restriction Sk(M' \cap M^*)$. 
Hence $\sigma_{M',P}(a) = \sigma_{M^*,P}(a) = \sigma_{N,N'}(a)$. 
So $\sigma_{M,P}(b) = \sigma_{M^*,P}(\sigma_{M,M^*}(b)) = 
\sigma_{M^*,P}(a) = \sigma_{N,N'}(a)$. 
Since $b \in x$, $\sigma_{M,P}(b) \in y$ by definition. 
So $\sigma_{N,N'}(a) \in y$.

Now suppose that $M' \cap \beta_{M',N} = N \cap \beta_{M',N}$. 
Then by Lemma 2.8, $\sigma_{M',N'} \restriction Sk(M' \cap N) = 
\sigma_{N,N'} \restriction Sk(M' \cap N)$. 
Since $a$ is in $Sk(M' \cap N)$, 
$\sigma_{N,N'}(a) = \sigma_{M',N'}(a) = 
\sigma_{M',N'}(\sigma_{M,M'}(b)) = \sigma_{M,N'}(b)$, which is in $y$ 
since $b \in x$.

Finally assume that $N \cap \beta_{M',N} \in Sk(M')$. 
Fix $N^* \in Sk(M')$ which is isomorphic to $N$ such that 
$N \cap \beta_{M',N} = N^* \cap \beta_{M',N}$. 
Let $L := \sigma_{M',M}(N^*)$. 
Then $a \in Sk(M' \cap N) = Sk(N \cap \beta_{M',N}) = 
Sk(N^* \cap \beta_{M',N})$, so $a \in Sk(N^*)$. 
Also $\sigma_{M',M} \restriction N^* = \sigma_{N^*,L}$. 
Hence $\sigma_{N^*,L}(a) = \sigma_{M',M}(a) = b$. 
By Lemma 2.8, $\sigma_{N,N'} \restriction Sk(N \cap N^*) = 
\sigma_{N^*,N'} \restriction Sk(N \cap N^*)$. 
Therefore $\sigma_{N,N'}(a) = \sigma_{N^*,N'}(a)$. 
So $\sigma_{N,N'}(a) = \sigma_{N^*,N'}(a) = 
\sigma_{N^*,N'}(\sigma_{M,M'}(b)) = 
\sigma_{N^*,N'}(\sigma_{L,N^*}(b)) = 
\sigma_{L,N'}(b)$, which is in $y$ since $b \in x$.
\end{proof}

Recall that a forcing poset $\q$ is \emph{strongly proper} if for all 
sufficiently large regular cardinals 
$\theta$ with $\q \in H(\theta)$, there are club many sets 
$N$ in $P_{\omega_1}(H(\theta))$ such that 
for all $p \in N \cap \q$ there exists $q \le p$ which is \emph{strongly $N$-generic}, which means 
that for any dense subset $D$ of the forcing poset $\q \cap N$, 
$D$ is predense below $q$ in $\q$ (\cite{mitchell}).
Strong properness implies properness, which in turn implies that $\omega_1$ 
is preserved.

\begin{proposition}
The forcing poset $\p$ is strongly proper.
\end{proposition}

\begin{proof}
Fix a regular cardinal $\theta > \omega_2$, and let $N^*$ be a countable 
elementary substructure of $H(\theta)$ satisfying that 
$\p$ and $\pi$ are in $N^*$ and 
$N := N^* \cap \omega_2 \in \mathcal X$. 
Clearly there are club many such sets $N^*$. 
Note that since $\pi \in N^*$, 
$Sk(N) = \pi[N] = N^* \cap H(\omega_2)$. 
In particular, $\p \cap N^* \subseteq Sk(N)$.

Let $p$ be a condition in $N^* \cap \p$. 
Define $q = (x_p,A_p \cup \{ N \})$. 
Then $q$ is a condition and $q \le p$. 
We will prove that $q$ is strongly $N^*$-generic. 
So let $D$ be a dense subset of $N^* \cap \p$, and we will show that 
$D$ is predense below $q$.

Fix $r \le q$, and we will 
find a condition $w$ in $D$ which is compatible with $r$. 
Since $N \in A_r$, $A_r \cap Sk(N)$ is a coherent adequate set by Lemma 2.7. 
Let $v = (x_r \cap Sk(N),A_r \cap Sk(N))$. 
Then $v$ is a condition in $\p$. 
Since $D$ is dense in $N^* \cap \p$, 
we can fix $w$ which is an extension of $v$ in $D$. 
Then $A_r \cap Sk(N) \subseteq A_w \subseteq Sk(N)$.

Let $C$ be the set 
$$
\{ M \in A_r : N \cap \omega_1 \le M \cap \omega_1 \} \cup 
\{ \sigma_{N,N'}(K) : N' \in A_r, \ N \cap \omega_1 = N' \cap \omega_1, \ 
K \in A_w \}.
$$
By Proposition 3.5, 
$C$ is a coherent adequate set which contains $A_r \cup A_w$. 
Let $y$ be the set 
$$
(x_r \setminus Sk(N)) \cup 
\{ \sigma_{N,N'}(a) : N' \in A_r, \ N \cap \omega_1 = N' \cap \omega_1, \ 
a \in x_w \}.
$$
Let $s := (y,C)$.

We claim that $s$ is a condition and $s \le r, w$, which completes the 
proof since $w$ is in $D$. 
If $a$ is in $x_w$, then $\sigma_{N,N}(a) = a$ is in $y$. 
And if $a$ is in $x_r$, then either $a$ is in $x_r \setminus Sk(N)$, and 
hence is in $y$ by definition, 
or else $a$ is in $x_w$, and hence is in 
$y$ as just noted. 
So $x_r$ and $x_w$ are subsets of $y$. 
Also $A_r$ and $A_w$ are subsets of $C$. 
Thus if $s$ is a condition then $s \le r, w$.

(1) We show that $y$ is a set of nonoverlapping triples. 
So let $a_0$ and $a_1$ be in $y$. 
Let $a_0 = \langle \alpha_0, \gamma_0, \beta_0 \rangle$ and 
$a_1 = \langle \alpha_1, \gamma_1, \beta_1 \rangle$. 
If $\alpha_0 \ne \alpha_1$ then $a_0$ and $a_1$ are nonoverlapping, 
so assume that $\alpha_0 = \alpha_1$. 
If $a_0$ and $a_1$ are both in $x_r \setminus Sk(N)$ then they are nonoverlapping 
since $r$ is a condition.

Suppose that $a_0 \in x_r \setminus Sk(N)$ 
and $a_1 = \sigma_{N,N'}(a)$ for some $a \in x_w$ and 
$N'$ in $A_r$ which is isomorphic to $N$. 
Since $\alpha_0 \in N'$, either $\gamma_0$ and $\beta_0$ are in $N'$ 
or $\sup(N' \cap \alpha_0) < \gamma_0$. 
In the latter case, $\beta_1 < \gamma_0$ and hence $a_0$ and $a_1$ 
are nonoverlapping. 
In the former case, $a_0$ is in $Sk(N') \cap x_r$. 
Hence $a^* := \sigma_{N',N}(a_0)$ is in $Sk(N) \cap x_r \subseteq x_w$. 
So $a^*$ and $a$ are nonoverlapping. 
Therefore their images under $\sigma_{N,N'}$, namely $a_0$ and $a_1$, 
are nonoverlapping.

Now suppose that $a_0 = \sigma_{N,N'}(a_0^*)$ and 
$a_1 = \sigma_{N,N^*}(a_1^*)$, where $a_0^*$ and $a_1^*$ are 
in $x_w$ and $N'$ and $N^*$ are isomorphic in $A_w$. 
If $N' = N^*$, then since $a_0^*$ and $a_1^*$ are nonoverlapping, so 
are their images under $\sigma_{N,N^*}$, namely $a_0$ and $a_1$. 
Suppose $N \ne N'$. 
By Lemma 2.8, fix $\beta \in N \cap \Lambda$ such that 
$\sigma_{N,N'} \restriction Sk(N \cap \beta) = 
\sigma_{N,N^*} \restriction Sk(N \cap \beta)$ is an isomorphism of 
$N \cap \beta$ to $N' \cap N^*$. 
But $\alpha_0 = \alpha_1$ implies that 
$\beta_{N',N^*} > \alpha_0$. 
Hence $a_0$ and $a_1$ are in $Sk(N' \cap N^*)$. 
Since $a_0^*$ and $a_1^*$ are nonoverlapping, their images under 
$\sigma_{N,N'} \restriction Sk(N \cap \beta)$, namely $a_0$ and $a_1$, 
are also nonoverlapping.

(2) We already noted that $C$ is a finite coherent adequate set.

(3) Let $M$ be in $C$ and $a$ in $y$, and we will show that $M$ and $a$ 
are nonoverlapping. 
If $M \cap \omega_1 \ge N \cap \omega_1$ and $a$ is in $x_r \setminus Sk(N)$, 
then we are done since $r$ is a condition. 
Let $a = \langle \alpha, \gamma, \beta \rangle$. 
If $\alpha \notin M$, then $a$ and $M$ are nonoverlapping, so 
assume that $\alpha \in M$. 
We will show that either $\gamma$ and $\beta$ are in $M$ or 
$\sup(M \cap \alpha) < \gamma$.

Suppose that $M \cap \omega_1 \ge N \cap \omega_1$ and 
$a = \sigma_{N,N'}(a^*)$ for some $N'$ in $A_r$ isomorphic to $N$ and 
some $a^*$ in $x_w$. 
Since $M \cap \omega_1 \ge N' \cap \omega_1$, either 
$N' \cap \beta_{N',M} \in Sk(M)$ or $N' \cap \beta_{N',M} = M \cap \beta_{N',M}$. 
But $\alpha \in M \cap N'$, so $\beta_{N',M} > \alpha$. 
So $\gamma$ and $\beta$ are in $N \cap \beta_{N',M}$ and hence 
in $M$.

Assume that $M = \sigma_{N,N'}(K)$, where $N' \in A_r$ 
is isomorphic to $N$ and $K \in A_w$, and $a \in x_r \setminus Sk(N)$. 
Since $M \subseteq N'$, $\alpha \in N'$. 
So either $\gamma$ and $\beta$ are in $N'$ or 
$\sup(N' \cap \alpha) < \gamma$. 
In the latter case, clearly $\sup(M \cap \alpha) < \gamma$ and we are done. 
Otherwise $a$ is a member of $Sk(N')$. 
So $b := \sigma_{N',N}(a) \in x_r \cap Sk(N) \subseteq x_w$. 
So $K$ and $b$ are nonoverlapping. 
Hence their images under $\sigma_{N,N'}$, namely $M$ and $a$, 
are nonoverlapping.

In the final case, suppose that 
$M = \sigma_{N,N'}(K)$, where $N' \in A_r$ is isomorphic to $N$ 
and $K \in A_w$, and 
$a = \sigma_{N,N^*}(b)$ for some $N^*$ in $A_r$ isomorphic to $N$ and 
some $b$ in $x_w$. 
So $K$ and $b$ are nonoverlapping. 
If $N' = N^*$, then the images of $K$ and $b$ under 
$\sigma_{N,N'}$, namely $M$ and $a$, are nonoverlapping. 
Otherwise by Lemma 2.8 we can 
fix $\beta \in N \cap \Lambda$ such that 
$\sigma_{N,N'} \restriction Sk(N \cap \beta) = 
\sigma_{N,N^*} \restriction Sk(N \cap \beta)$ is an isomorphism of 
$N \cap \beta$ to $N' \cap N^*$. 
As $\alpha \in M$, $\alpha$ is in $N' \cap N^*$. 
Since $N' \cap N^*$ is an initial segment of $N'$ and $N^*$, 
$a \in Sk(N' \cap N^*)$. 
Hence $b$ is in $Sk(N \cap \beta)$. 
Therefore $a = \sigma_{N,N^*}(b) = \sigma_{N,N'}(b)$. 
Thus $a$ and $M$ are the images of $b$ and $K$ under 
$\sigma_{N,N'}$, and $b$ and $K$ are nonoverlapping. 
So $a$ and $M$ are nonoverlapping.

(4) By Lemma 4.2 it suffices to show that $y$ is equal to the set 
$$
x_r \cup x_w \cup \{ \sigma_{M,M'}(a) : M, M' \in C, \ 
M \cap \omega_1 = M' \cap \omega_1, \ 
a \in (x_r \cup x_w) \cap Sk(M) \}.
$$
Clearly $y$ is a subset of this set. 
It was noted above that $x_r \cup x_w \subseteq y$. 
Suppose that $M$ and $M'$ are 
isomorphic sets in $C$ and $a \in (x_r \cup x_w) \cap Sk(M)$. 
We will show that $a^* := \sigma_{M,M'}(a) \in y$.

Suppose that $M \cap \omega_1 > N \cap \omega_1$. 
Then also $M' \cap \omega_1 > N \cap \omega_1$. 
If $a$ is in $x_r$, then we are done since $r$ is a condition. 
Suppose that $a$ is in $x_w$. 
Fix $N^*$ in $Sk(M)$ which is isomorphic to $N$ such that 
$N \cap \beta_{M,N} = N^* \cap \beta_{M,N}$. 
Then $a \in Sk(N \cap \beta_{M,N}) = Sk(N^* \cap \beta_{M,N})$. 
Let $P := \sigma_{M,M'}(N^*)$. 
So $\sigma_{M,M'} \restriction Sk(N^*) = \sigma_{N^*,P}$. 
By Lemma 2.8, $\sigma_{M,M'}(a) = \sigma_{N^*,P}(a) = 
\sigma_{N,P}(a)$, which is in $y$ by definition.

Now assume that $M \cap \omega_1 = N \cap \omega_1$. 
Then $M$, $M'$, and $N$ are all isomorphic. 
If $a \in x_r$ then we are done since $r$ is a condition. 
Suppose that $a \in x_w$. 
Since $a \in Sk(M) \cap Sk(N) = Sk(M \cap N)$, 
by Lemma 2.8, $\sigma_{M,M'}(a) = \sigma_{N,M'}(a)$, 
which is in $y$ by definition. 

Finally, suppose that $M \cap \omega_1 < N \cap \omega_1$. 
By the definition of $C$, 
$M = \sigma_{N,N'}(K)$ for some $N'$ in $A_r$ 
which is isomorphic to $N$ and some $K \in A_w$. 
Then also $M' = \sigma_{N,N^*}(P)$ for some $N^*$ in $A_r$ which is 
isomorphic to $N$ and some $P \in A_w$. 
Since $a$ is in $Sk(M)$, $a$ is in $Sk(N')$. 
We claim that $b := \sigma_{N',N}(a)$ is in $x_w$. 
If $a \in x_r$, then since $r$ is a condition, $b$ is in 
$x_r \cap Sk(N)$ and hence in $x_w$. 
Otherwise $a$ is in $x_w$ and hence in $Sk(N') \cap Sk(N) = 
Sk(N' \cap N)$. 
But $\sigma_{N',N} \restriction Sk(N' \cap N)$ is the 
identity, so $b = a$.

We have that $\sigma_{N',N} \restriction Sk(M) = \sigma_{M,K}$ 
and $\sigma_{N^*,N} \restriction Sk(M') = \sigma_{M',P}$. 
And $\sigma_{M,M'} = \sigma_{P,M'} \circ \sigma_{K,P} \circ 
\sigma_{M,K} = 
\sigma_{N,N^*} \circ \sigma_{K,P} \circ (\sigma_{N',N} \restriction Sk(M))$. 
So $\sigma_{M,M'}(a) = 
\sigma_{N,N^*}(\sigma_{K,P}(\sigma_{N',N}(a))) = 
\sigma_{N,N^*}(\sigma_{K,P}(b))$. 
Since $b \in x_w$ and $K$ and $P$ are in $A_w$, 
$\sigma_{K,P}(b)$ is in $x_w$. 
Hence $\sigma_{M,M'}(a) = \sigma_{N,N^*}(\sigma_{K,P}(b))$ 
is in $y$ by definition.
\end{proof}

\begin{proposition}
The forcing poset $\p$ is $\omega_2$-c.c.
\end{proposition}

\begin{proof}
Fix $\theta > \omega_2$ regular and let 
$N^*$ be an elementary substructure of $H(\theta)$ of size $\omega_1$ 
such that $\pi$, $\mathcal X$, $\Lambda$, and $\p$ are in $N^*$ and 
$\beta := N^* \cap \omega_2 \in \Lambda$. 
Since $\pi \in N^*$, $N^* \cap H(\omega_2) = \pi[N^* \cap \omega_2] 
= \pi[\beta] = Sk(\beta)$. 
In particular, $N^* \cap \p \subseteq Sk(\beta)$. 
Note that since $\mathcal X \cap P(\beta) \subseteq Sk(\beta)$, 
$N^* \cap \mathcal X = P(\beta) \cap \mathcal X = Sk(\beta) \cap \mathcal X$.

We will prove that the empty condition is $N^*$-generic. 
This implies that $\p$ is $\omega_2$-c.c.\ by the following argument. 
Suppose for a contradiction that $\p$ has a maximal antichain $S$ of size 
at least $\omega_2$. 
By elementarity we may assume that $S$ is in $N^*$. 
Since $N^*$ has size $\omega_1$, we can fix a condition 
$s \in S \setminus N^*$. 
Let $D$ be the set of conditions which are below some member of $S$. 
Then $D$ is dense and $D$ is in $N^*$. 
Since the empty condition is $N^*$-generic, $N^* \cap D$ is predense in $\p$. 
So $s$ is compatible with some member of $N^* \cap D$. 
By elementarity and the definition of $D$, $s$ is compatible with some 
member of $N^* \cap S$, which contradicts that $S$ 
is an antichain.

Note that since $2^\omega = \omega_1$ and $\omega_1 \subseteq N^*$, 
$H(\omega_1) \subseteq N^*$. 
Fix a dense open set $D$ in $N^*$, and we will show that 
$D \cap N^*$ is predense in $\p$. 
Let $p$ be a given condition. 
Extend $p$ to $q$ which is in $D$.

Let $A^- := \{ M \in A_q : M \subseteq \beta \}$.  
Let $A^+ := \{ M \in A_q : M \setminus \beta \ne \emptyset \} = 
\{ M_0, \ldots, M_k \}$. 
Since $\Lambda \in N^*$, $\Lambda \cap \beta$ is cofinal in $\beta$. 
Fix $\beta^*$ in $\Lambda \cap \beta$ such that for all $M \in A_q$, 
$\sup(M \cap \beta) < \beta^*$, and for all 
$\langle \alpha, \gamma, \zeta \rangle$ 
in $x_r \cap Sk(\beta)$, $\alpha < \beta^*$. 
Let $R$ be the set of pairs $\langle i, j \rangle$ in $k+1$ such that 
$M_i \in Sk(M_j)$.
Note that the objects $A^-$, 
$M_0 \cap \beta, \ldots, M_k \cap \beta$, $\beta^*$, and $R$ are in $N^*$.

For each $i = 0, \ldots, k$, let $\overline{\mathfrak M}_i$ denote the 
transitive collapse of the structure 
$\mathfrak M_i = (Sk(M_i),\in,\pi_{M_i},{\mathcal X}_{M_i},\Lambda_{M_i})$. 
And for each $\langle i, j \rangle$ in $R$, let 
$J_{\langle i, j \rangle} := \sigma_{M_j}(M_i)$. 
Note that each $\overline{\mathfrak M}_i$ is in $H(\omega_1)$ and 
hence in $N^*$, and therefore each $J_{\langle i, j \rangle}$ is in $N^*$.

Let $a_0, \ldots, a_m$ enumerate the triples in $x_q$ whose first component 
is larger than $\beta$. 
Let $S$ be the set of pairs $\langle i, j \rangle$ where 
$i \le m$, $j \le k$, and $a_i \in Sk(M_j)$. 
For each $\langle i, j \rangle$ in $S$, 
let $b_{\langle i,j \rangle} = \sigma_{M_j}(a_i)$.

As noted above, the following parameters all belong to $N^*$: 
$x_q \cap Sk(\beta)$, $A^-$, $D$, $M_0 \cap \beta, \ldots, M_k \cap \beta$, 
$\pi$, $\mathcal X$, $\Lambda$, $\mathfrak M_0, \ldots, \mathfrak M_k$, 
$R$, $J_{\langle i, j \rangle}$ for each $\langle i, j \rangle \in R$, 
$\beta^*$, $S$, and $b_{\langle i, j \rangle}$ for each $\langle i, j \rangle \in S$. 
Let $\varphi_{x_0,\ldots,x_k,y_0,\ldots,y_m}$ be the formula in the 
language of set theory with constants for these parameters which expresses 
the following:
\begin{enumerate}
\item the pair 
$$
((x_q \cap Sk(\beta)) \cup \{ y_0, \ldots, y_m \}, 
A^- \cup \{ x_0,\ldots,x_k \})
$$
is in $D$;
\item for each $i = 0, \ldots, k$, $x_i \cap \beta^* = M_i \cap \beta$;
\item for each $i = 0, \ldots, k$, the transitive collapse of 
$(Sk(x_i),\in,\pi_{x_i},{\mathcal X}_{x_i},\Lambda_{x_i})$ 
is equal to $\overline{\mathfrak M}_i$;
\item for each $i, j < k+1$, $x_i \in Sk(x_j)$ iff $\langle i, j \rangle \in R$, 
and in that case, $\sigma_{x_j}(x_i) = J_{\langle i, j \rangle}$;
\item for each $i = 0, \ldots, m$, the first component of $y_i$ is above $\beta^*$;
\item for each $i \le m$ and $j \le k$, 
$y_i \in Sk(x_j)$ iff $\langle i, j \rangle \in S$, and in that case, 
$\sigma_{x_j}(y_i) = b_{\langle i, j \rangle}$.
\end{enumerate}
Note that $H(\theta) \models \varphi[M_0, \ldots, M_k, a_0, \ldots, a_m]$. 
By elementarity we can find 
$M_0', \ldots, M_k'$ and $a_0', \ldots, a_m'$ in $N^*$ 
such that 
$H(\theta) \models \varphi[M_0',\ldots,M_k',a_0',\ldots,a_m']$.

Let $w$ denote the pair 
$$
((x_q \cap Sk(\beta)) \cup \{ a_0', \ldots, a_m' \}, 
A^- \cup \{ M_0',\ldots,M_k'\}).
$$
Then $w$ is in $D$ by (1).

Let us verify that the assumptions of Proposition 3.6 hold for the map 
which sends $M$ to $M'$ for each $M \in A^+$. 
Let $M$ and $K$ be in $A^+$. 
(3) implies that $\mathfrak M$ and 
${\mathfrak M}'$ have the same 
transitive collapse and hence are isomorphic, and (2) implies that 
$M' \cap \beta^* = M \cap \beta = M \cap \beta^*$. 
Let $M = M_j$ and $K = M_i$ for $i, j \le k$. 
By (4), $K \in Sk(M)$ iff $\langle i, j \rangle \in R$ iff $K' \in Sk(M')$, and in that 
case, $\sigma_{M}(K) = J_{\langle i, j \rangle}$ by definition and 
$\sigma_{M'}(K') = J_{\langle i, j \rangle}$ by (4).
But $\sigma_{M,M'} = \sigma_{M'}^{-1} \circ \sigma_M$. 
So $\sigma_{M,M'}(K) = \sigma_{M'}^{-1}(\sigma_M(K)) = 
\sigma_{M'}^{-1}(J_{\langle i, j \rangle}) = K'$. 
Finally, $A^- \cup \{ M_0', \ldots, M_k' \}$ is a coherent adequate set by (1).
It follows by Proposition 3.6 that the set 
$$
C := A_q \cup \{ M' : M \in A^+ \}
$$
is a coherent adequate set.

By (6), for each $i \le m$ and $j \le k$, $a_i \in Sk(M_j)$ iff 
$\langle i, j \rangle \in J$ iff $a_i' \in Sk(M_j')$. 
Also if $a_i \in Sk(M_j)$, then 
$\sigma_{M_j,M_j'}(a_i) = \sigma_{M_j^{-1}}(\sigma_{M_j}(a_i)) = 
\sigma_{M_j^{-1}}(b_{\langle i, j \rangle}) = a_j'$. 
So $\sigma_{M_i,M_i'}(a_j) = a_j'$. 
Let 
$$
y := x_q \cup \{ a_j' : j = 0, \ldots, m \}.
$$
By (5) the first component of each $a_j'$ is above $\beta^*$. 
Hence any element of $y$ is in $x_q \cap Sk(\beta)$, 
$\{ a_j' : j = 0, \ldots, m \}$, or $x_q \setminus Sk(\beta)$ depending on 
whether its first component is in $[0,\beta^*)$, 
$[\beta^*,\beta)$, or $[\beta^*,\omega_2)$.

We claim that $s = (y,C)$ is a condition. 
Then clearly $s \le r, w$, and since $w$ is in $D$, we are done.

(1) Let $\langle \alpha, \gamma, \zeta \rangle$ and 
$\langle \alpha', \gamma', \zeta' \rangle$ be in $y$, and we will show 
that they are nonoverlapping. 
If these triples are either both in $x_q$ or both in $x_w$, then we are done. 
Otherwise we may assume that $\langle \alpha, \gamma, \zeta \rangle$ 
is equal to $a_i$ for some $i = 0, \ldots, m$ and 
$\langle \alpha', \gamma', \zeta' \rangle$ is equal to $a_j'$ for some 
$j = 0, \ldots, m$. 
Then $\alpha' < \beta \le \alpha$, so these triples are nonoverlapping.

(2) The set $C$ is a finite coherent adequate set as previously noted.

(3) Let $M$ be in $C$ and $\langle \alpha, \gamma, \zeta \rangle$ in $y$, and 
we will show that they are nonoverlapping. 
If $\alpha$ is not in $M$, then we are done, so 
assume that $\alpha \in M$. 
If these objects are either both in $q$ or both in $w$, then we are done. 
Assume that $M \in C \setminus Sk(\beta)$ and 
$\langle \alpha, \gamma, \zeta \rangle \in y \cap Sk(\beta)$. 
Since $\alpha \in M \cap \beta$, $\alpha$ is in $M' \cap \beta^*$. 
But the triple and $M'$ are nonoverlapping, and since $\alpha < \beta^*$ 
this clearly implies that the triple and $M$ are nonoverlapping. 
Next assume that $M \in C \cap Sk(\beta)$ and 
$\langle \alpha, \gamma, \zeta \rangle \in y \setminus Sk(\beta)$. 
Then $\alpha \ge \beta$. 
But this is impossible since $M \subseteq \beta$.

(4) Let $M$ and $K$ be isomorphic sets 
in $C$ and $a \in y \cap Sk(M)$. 
We will show that $\sigma_{M,K}(a) \in y$. 
Let $a = \langle \alpha, \gamma, \zeta \rangle$.

Suppose that $M \in A_q$. 
Then $\alpha \notin [\beta^*,\beta)$, hence $a \in x_q$. 
If $K$ is in $A_q$ we are done; otherwise $K = P'$ for some $P \in A^+$. 
Then $\sigma_{M,P}(a) \in x_q \cap Sk(P)$. 
Assume that $\sigma_{M,P}(\alpha) \ge \beta$. 
Then $\sigma_{M,P}(a) = a_i$ for some $i \le m$. 
So $\sigma_{P,P'}(a) = a_i'$. 
So $\sigma_{M,K}(a) = \sigma_{P,P'}(\sigma_{M,P}(a)) = a_i' \in y$. 
Now assume that $\sigma_{M,P}(\alpha) < \beta^*$. 
Then $\sigma_{P,P'}(\sigma_{M,P}(a)) = 
\sigma_{M,P}(a)$ since $\sigma_{P,P'} \restriction \beta^*$ is 
the identity. 
So $\sigma_{M,K}(a) = \sigma_{P,P'}(\sigma_{M,P}(a)) = \sigma_{M,P}(a)$, 
which is in $y$.

Now suppose that $M = L'$ for some $L \in A^+$. 
Then $M \in A_w$.  
So $a$ is in $(x_q \cap Sk(\beta)) \cup \{ a_0', \ldots, a_m' \} = x_w$. 
If $K \in A_w$ then we are done since $w$ is a condition. 
Otherwise $K \in C \setminus Sk(\beta)$. 
Then $K' \in A_w$, so $\sigma_{M,K'}(a) \in x_w$. 
If $\sigma_{M,K'}(a) < \beta^*$, then 
$\sigma_{K',K}(\sigma_{M,K'}(a)) = \sigma_{M,K'}(a)$ since 
$\sigma_{K',K} \restriction \beta^*$ is the identity. 
Hence $\sigma_{M,K}(a) = \sigma_{K',K}(\sigma_{M,K'}(a)) = 
\sigma_{M,K'}(a)$, which is in $y$. 
Otherwise $\sigma_{M,K'}(a)$ is equal to $a_i'$ for some $i = 0, \ldots, m$. 
So $a_i' \in Sk(K')$, which implies that $a_i \in Sk(K)$ and 
$\sigma_{K,K'}(a_i) = a_i'$. 
Hence $\sigma_{M,K}(a) = \sigma_{K',K}(\sigma_{M,K'}(a)) = 
\sigma_{K',K}(a_i') = a_i$, which is in $y$.
\end{proof}

This completes the proof that $\p$ preserves cardinals.

\bigskip

Recall that 
for each $\alpha \in \Lambda$, $\dot c_\alpha$ is a $\p$-name such 
that $\p$ forces
$$
\dot c_\alpha = \{ \gamma : \exists p \in \dot G \ \exists \beta \ 
\langle \alpha, \gamma, \beta \rangle \in x_p \}.
$$
We will show that $\p$ forces that 
$\dot c_\alpha$ is a cofinal subset of $\alpha$. 
Property (3) in the definition of $\p$ will imply that 
$\dot c_\alpha$ is forced to have order type $\omega_1$. 
Property (4) will imply that $\p$ forces that whenever $\xi$ 
is a common limit point of $\dot c_\alpha$ and 
$\dot c_{\alpha'}$, then 
$\dot c_\alpha \cap \xi = \dot c_{\alpha'} \cap \xi$. 

\begin{lemma}
For each $\alpha \in \Lambda$, $\p$ forces that $\dot c_\alpha$ is a cofinal 
subset of $\alpha$ with order type $\omega_1$.
\end{lemma}

\begin{proof}
First we show that $\dot c_\alpha$ is forced to be a cofinal 
subset of $\alpha$. 
Let $p$ be a condition and $\delta < \alpha$. 
Choose an ordinal $\gamma$ with $\delta < \gamma < \alpha$ 
such that for all $M \in A_p$, $\sup(M \cap \alpha) < \gamma$, and for 
all triples in $x_p$ of the form $\langle \alpha, \tau, \beta \rangle$, 
$\tau$ and $\beta$ are less than $\gamma$. 
Define $q = (x_p \cup \{ \langle \alpha, \gamma, \gamma+1 \rangle \}, A_p )$. 
It is easy to check that $q$ is a condition, and clearly $q \le p$. 
Also $q$ forces that $\dot c_\alpha \setminus \delta$ is nonempty. 
Thus $\p$ forces that $\dot c_\alpha$ is a cofinal subset of $\alpha$.

Suppose for a contradiction that a condition $p$ forces that $\dot c_\alpha$ 
has order type greater than $\omega_1$. 
Extending $p$ if necessary, assume that for some $\delta < \alpha$, 
$p$ forces that $\dot c_\alpha \cap \delta$ has size $\omega_1$. 
Fix $M$ in $\mathcal X$ such that $p$, $\alpha$, and $\delta$ are in $Sk(M)$. 
Then easily $q = (x_p,A_p \cup \{M\})$ is a condition. 
Since $q$ forces that $\dot c_\alpha \cap \delta$ is uncountable, 
we can extend $q$ to $r$ such that for some triple 
$\langle \alpha, \gamma, \beta \rangle$ in $x_r$, 
$\gamma$ is $\delta \setminus M$. 
Since $M \in A_r$ and $\alpha \in M$, $\sup(M \cap \alpha) < \gamma$, 
which contradicts that $\delta \in M$.
\end{proof}

Now we prove that the sequence of $\dot c_\alpha$'s is coherent. 
Namely, we will show that $\p$ forces that 
whenever $\xi$ is a common limit point of 
$\dot c_\alpha$ and $\dot c_{\alpha'}$, then 
$\dot c_\alpha \cap \xi = \dot c_{\alpha'} \cap \xi$.

\begin{lemma}
Let $\alpha$ be in $\Lambda$, $\xi < \alpha$, and suppose that 
$p$ is a condition which forces that 
$\xi$ is a limit point of $\dot c_\alpha$. 
Then there is $M \in A_p$ such that $\alpha \in M$ and 
$\sup(M \cap \alpha) = \xi$.
\end{lemma}

\begin{proof}
Note that for all $q \le p$, 
since $q$ forces that $\xi$ is a limit point of $\dot c_\alpha$, if 
$\langle \alpha, \gamma, \beta \rangle \in x_q$ and $\gamma < \xi$, 
then $\beta < \xi$. 
Suppose for a contradiction that for all $M \in A_p$, if 
$\alpha \in M$ then $\sup(M \cap \alpha) \ne \xi$. 

We claim that if $M \in A_p$, $\alpha \in M$, and $\sup(M \cap \xi) < \xi$, 
then $\sup(M \cap \alpha) < \xi$. 
Otherwise fix a counterexample $M$. 
Then $\alpha \in M$, $\sup(M \cap \xi) < \xi$, and $\sup(M \cap \alpha) \ge \xi$. 
Since $\xi$ is forced to be a limit point of $\dot c_\alpha$, we can find 
$q \le p$ and $\gamma, \beta < \xi$ such that 
$\langle \alpha, \gamma, \beta \rangle \in x_q$ and 
$\sup(M \cap \xi) < \gamma$. 
Then $\gamma$ and $\beta$ are not in $M$, but 
$\sup(M \cap \alpha) \ge \xi > \gamma$, which contradicts 
that $q$ is a condition.

It follows from the claim that $A$ 
is the union of the sets 
$A_0$, $A_1$, and $A_2$ 
defined by 
$$
A_0 = \{ M \in A_p : \alpha \notin M \},
$$
$$
A_1 = \{ M \in A_p : \alpha \in M, \ \sup(M \cap \alpha) < \xi \},
$$
$$
A_2 = \{ M \in A_p : \alpha \in M, \ \sup(M \cap \xi) = \xi \}.
$$
Since we are assuming that there is no $M$ in $A_p$ with $\alpha \in M$ and 
$\sup(M \cap \alpha) = \xi$, 
any every set in $A_2$ meets the interval $[\xi,\alpha)$. 
Observe that if $N \in A_1$ and $M \in A_2$, then since 
$\alpha \in M \cap N$, $\beta_{M,N} > \alpha$; hence 
$\sup(N \cap \alpha) < \xi < \sup(M \cap \alpha)$ implies that 
$N \cap \beta_{M,N} \in Sk(M)$.

Fix $M$ in $A_2$ such that $M \cap \omega_1$ is minimal. 
Let $\tau = \min(M \setminus \xi)$. 
Then $\xi \le \tau < \alpha$. 
Since $\sup(M \cap \xi) = \xi$, we can 
fix $\gamma < \xi$ in $M$ such that for all $N \in A_1$, 
$\sup(N \cap \alpha) < \gamma$, and 
for all $\langle \alpha,\zeta,\beta \rangle \in x_p$, if $\zeta < \xi$ then 
$\zeta, \beta < \gamma$.

Let $y$ be the set of triples of the form 
$\sigma_{N,N'}(\langle \alpha, \gamma, \tau \rangle)$, where 
$N$ and $N'$ are isomorphic sets in $A_p$ and 
$\langle \alpha, \gamma, \tau \rangle \in Sk(N)$. 
Let $q = (x_p \cup y, A_p)$. 
We claim that $q$ is a condition. 
Then clearly $q \le p$ and $q$ forces that $\xi$ is not a limit point of 
$\dot c_\alpha$, which is a contradiction.

Let us note that $\langle \alpha, \gamma, \tau \rangle$ is nonoverlapping with 
every triple in $x_p$. 
Let $\langle \alpha, \gamma', \beta' \rangle$ be in $x_p$. 
If $\gamma' < \xi$ then $\gamma'$ and $\beta'$ are below $\gamma$, so 
we are done. 
Suppose that $\gamma' \ge \xi$. 
Since $M \in A_p$, either $\gamma'$ and $\beta'$ are in $M$ or 
$\sup(M \cap \alpha) < \gamma'$. 
In the former case, $\tau = \min(M \setminus \xi) \le \gamma'$. 
In the latter case, $\tau < \sup(M \cap \alpha) < \gamma'$. 
In either case, $\tau \le \gamma'$, which implies that 
$[\gamma,\tau) \cap [\gamma',\beta') = \emptyset$.

Next we claim that if $K \in A_p$ then $K$ and 
$\langle \alpha, \gamma, \tau \rangle$ are nonoverlapping. 
If $\alpha$ is not in $K$ then we are done, so assume that $\alpha \in K$. 
Then either $K \in A_1$ or $K \in A_2$. 
If $K \in A_1$, then $\sup(K \cap \alpha) < \gamma$ by the choice of $\gamma$. 
If $K \in A_2$, then since $M \cap \omega_1 \le K \cap \omega_1$, 
either $M \cap \beta_{K,M} \in Sk(K)$ or 
$M \cap \beta_{K,M} = K \cap \beta_{K,M}$. 
In either case, $M \cap \beta_{K,M} \subseteq K$. 
But since $\alpha \in K \cap M$, $\beta_{K,M} > \alpha$. 
So $\gamma$ and $\tau$ are in $K$.

Now we prove that $q$ is a condition.

(1) Consider a triple $\langle \alpha', \gamma', \beta' \rangle$ in $x_p$ and 
a triple $\sigma_{N,N'}(\langle \alpha, \gamma, \tau \rangle)$, where $N$ and 
$N'$ are isomorphic in $A_p$ and 
$\langle \alpha, \gamma, \tau \rangle \in Sk(N)$. 
If $\alpha' \ne \sigma_{N,N'}(\alpha)$ then we are done, so assume that 
$\alpha' = \sigma_{N,N'}(\alpha)$. 
If $\gamma'$ and $\beta'$ are not in $Sk(N')$, 
then $\sup(N' \cap \alpha') < \gamma'$, so clearly the triples are nonoverlapping. 
Otherwise $\gamma'$ and $\beta'$ are both in $Sk(N')$. 
Then $\langle \alpha', \gamma', \beta' \rangle \in x_p \cap Sk(N')$, 
so $\sigma_{N',N}(\langle \alpha',\gamma',\beta' \rangle)$ is in $x_p$. 
By the comments above, 
$\sigma_{N',N}(\langle \alpha',\gamma',\beta' \rangle)$ and 
$\langle \alpha, \gamma, \tau \rangle$ are nonoverlapping. 
Hence the images of these triples under $\sigma_{N,N'}$ 
are nonoverlapping and we are done.

Now consider $\sigma_{N_0,N'}(\langle \alpha, \gamma, \tau \rangle)$ and 
$\sigma_{N_1,N^*}(\langle \alpha, \gamma, \tau \rangle)$, where 
$N_0$ and $N'$ are isomorphic in $A_p$ and 
$\langle \alpha, \gamma, \tau \rangle \in Sk(N_0)$, 
and $N_1$ and $N^*$ are isomorphic in $A_p$ and 
$\langle \alpha, \gamma, \tau \rangle \in Sk(N_1)$. 
If $\sigma_{N_0,N'}(\alpha) \ne \sigma_{N_1,N^*}(\alpha)$ then the 
triples are nonoverlapping, so assume that 
$\alpha^* := \sigma_{N_0,N'}(\alpha) = \sigma_{N_1,N^*}(\alpha)$. 
Then $\beta_{N_0,N_1} > \alpha$ and 
$\beta_{N',N^*} > \alpha^*$.

We will show that 
$\sigma_{N_0,N'}(\langle \alpha, \gamma, \tau \rangle) = 
\sigma_{N_1,N^*}(\langle \alpha, \gamma, \tau \rangle)$. 
By symmetry it suffices to consider the cases when 
$N_0 \cap \beta_{N_0,N_1} \in Sk(N_1)$ and 
$N_0 \cap \beta_{N_0,N_1} = N_1 \cap \beta_{N_0,N_1}$. 
Suppose the former case. 
Then also $N' \cap \beta_{N',N^*} \in Sk(N^*)$. 
Fix $N_0^*$ in $Sk(N_1) \cap A_p$ which is isomorphic to $N_0$ such that 
$N_0 \cap \beta_{N_0,N_1} = N_0^* \cap \beta_{N_0,N_1}$. 
Then $\langle \alpha, \gamma, \tau \rangle \in Sk(N_0^*)$. 
Also fix $P \in Sk(N^*) \cap A_p$ which is isomorphic to $N'$ such that 
$N' \cap \beta_{N',N^*} = P \cap \beta_{N',N^*}$. 
Since $\beta_{N',N^*} > \alpha^*$, $\alpha^* \in P$.

Since $\sigma_{N_1,N^*}(\alpha) = \alpha^*$, 
$\alpha^* \in P \cap \sigma_{N_1,N^*}(N_0^*)$. 
As $P$ and $\sigma_{N_1,N^*}(N_0^*)$ are isomorphic and are 
in the adequate set $A_p$, it follows 
that $P \cap \alpha^* = \sigma_{N_1,N^*}(N_0^*) \cap \alpha^*$. 
Now $\sigma_{N_0,N'} \restriction \alpha$ is the unique order preserving 
map from $N_0 \cap \alpha = N_0^* \cap \alpha$ onto 
$N' \cap \alpha^* = P \cap \alpha^* = \sigma_{N_1,N^*}(N_0^*) \cap \alpha^*$. 
But also $\sigma_{N_1,N^*} \restriction (N_0^* \cap \alpha)$ 
is an order preserving map from $N_0^* \cap \alpha$ onto 
$\sigma_{N_1,N^*}(N_0^*) \cap \alpha^*$. 
It follows that $\sigma_{N_0,N'} \restriction \alpha = 
\sigma_{N_1,N^*} \restriction (N_0^* \cap \alpha)$. 
In particular, $\sigma_{N_0,N'}(\langle \alpha, \gamma, \tau \rangle) = 
\sigma_{N_1,N^*}(\langle \alpha, \gamma, \tau \rangle)$.

Now suppose that 
$N_0 \cap \beta_{N_0,N_1} = N_1 \cap \beta_{N_0,N_1}$. 
Then also $N' \cap \beta_{N',N^*} = N^* \cap \beta_{N',N^*}$. 
In particular, $N_0 \cap \alpha = N_1 \cap \alpha$ and 
$N' \cap \alpha^* = N^* \cap \alpha^*$. 
But $\sigma_{N_0,N'} \restriction \alpha$ is the unique order preserving 
map from $N_0 \cap \alpha$ onto $N' \cap \alpha^*$, and 
$\sigma_{N_1,N^*} \restriction \alpha$ is the unique order preserving 
map from $N_1 \cap \alpha$ onto $N^* \cap \alpha$. 
Hence $\sigma_{N_0,N'} \restriction \alpha = 
\sigma_{N_1,N^*} \restriction \alpha$. 
So $\sigma_{N_0,N'}(\langle \alpha, \gamma, \tau \rangle) = 
\sigma_{N_1,N^*}(\langle \alpha, \gamma, \tau \rangle)$.

(2) is immediate.

(3) Let $K$ be in $A_p$ and consider 
$\langle \alpha^*, \gamma^*, \tau^* \rangle := 
\sigma_{N,N'}(\langle \alpha, \gamma, \tau \rangle)$, where 
$N$ and $N'$ are isomorphic sets in $A_p$ 
and $\langle \alpha, \gamma, \tau \rangle$ is in $Sk(N)$. 
We will prove that $K$ and 
$\langle \alpha^*, \gamma^*, \tau^* \rangle$ 
are nonoverlapping. 
If $\alpha^*$ is not in $K$, then we are done, so assume that 
$\alpha^* \in K$.
Then $\beta_{K,N'} > \alpha^*$.

If $N' \cap \beta_{K,N'}$ is either in $Sk(K)$ or equal to $K \cap \beta_{K,N'}$ 
then $\gamma'$ and $\tau'$ are in $K$ and we are done. 
So assume that $K \cap \beta_{K,N'} \in Sk(N')$. 
Then there is $K^*$ in $Sk(N') \cap A_p$ 
which is isomorphic to $K$ such that 
$K^* \cap \beta_{K,N'} = K \cap \beta_{K,N'}$. 
Since $\alpha^* < \beta_{K,N'}$, it suffices to show that $K^*$ and 
$\langle \alpha^*, \gamma^*, \tau^* \rangle$ are nonoverlapping. 
But $L := \sigma_{N',N}(K^*)$ is in $A_p$, and we showed above that 
$L$ is nonoverlapping with $\langle \alpha, \gamma, \tau \rangle$. 
Therefore the images of $L$ and $\langle \alpha, \gamma, \tau \rangle$ 
under $\sigma_{N,N'}$, namely $K^*$ and 
$\langle \alpha^*, \gamma^*, \tau^* \rangle$, are nonoverlapping.

(4) By Lemma 4.2 it suffices to show that 
$x_p \cup y$ is equal to 
$$
x_p^* \cup 
\{ \sigma_{N,N'}(a) : N, N' \in A_p, \ 
N \cap \omega_1 = N' \cap \omega_1, \ 
a \in x_p^* \cap Sk(N) \},
$$
where $x_p^* = x_p \cup \{ \langle \alpha, \gamma, \tau \rangle \}$. 
Clearly $x_p \cup y$ is included in the second set by definition, and 
$x_p^* \subseteq x_p \cup y$. 
Consider $a \in x_p \cup \{ \langle \alpha, \gamma, \tau \rangle \}$ and isomorphic 
$N$ and $N'$ in $A_p$ with $a \in Sk(N)$. 
If $a \in x_p$ then $\sigma_{N,N'}(a) \in x_p$ since $p$ is a condition. 
Otherwise $a = \langle \alpha, \gamma, \beta \rangle$, and 
$\sigma_{N,N'}(a) \in y$ by the definition of $y$.
\end{proof}

\begin{proposition}
Let $\alpha$ and $\alpha'$ be distinct ordinals in $\Lambda$. 
Then $\p$ forces that whenever $\xi$ is a common limit point of 
$\dot c_\alpha$ and $\dot c_{\alpha'}$, 
$\dot c_\alpha \cap \xi = \dot c_{\alpha'} \cap \xi$.
\end{proposition}

\begin{proof}
Let $p$ be a condition which forces that $\xi$ is a common limit point 
of $\dot c_\alpha$ and $\dot c_{\alpha'}$. 
Then by the previous lemma, there are $M$ and $M'$ in $A_p$ such that 
$\alpha \in M$ and $\sup(M \cap \alpha) = \xi$, and 
$\alpha' \in M'$ and $\sup(M' \cap \alpha') = \xi$. 
Since $\xi$ is a common limit point of $M$ and $M'$, 
$\xi < \beta_{M,M'}$. 
It is not possible that $M \cap \beta_{M,M'} \in Sk(M')$, since 
in that case $\xi$, which is a limit point of $M \cap \beta_{M,M'}$, would be in $M'$. 
Similarly, $M' \cap \beta_{M,M'}$ is not in $Sk(M)$. 
So $M \cap \beta_{M,M'} = M' \cap \beta_{M,M'}$.
It follows that $M$ and $M'$ are isomorphic. 
Also $\sigma_{M,M'} \restriction M \cap \beta_{M,M'}$ is the identity 
and $\sigma_{M,M'}(\alpha) = \alpha'$. 

Suppose that $q \le p$ and $q$ forces that $\gamma$ is in $\dot c_\alpha \cap \xi$. 
Extending $q$ if necessary, assume that 
$\langle \alpha, \gamma, \beta \rangle \in x_q$ 
for some $\beta$. 
Since $\gamma < \xi = \sup(M \cap \alpha)$, $\gamma$ and $\beta$ are in $M$. 
So $\sigma_{M,M'}(\langle \alpha, \gamma, \beta \rangle) = 
\langle \alpha', \gamma, \beta \rangle$ is in $x_q$. 
Hence $q$ forces that $\gamma$ is in $\dot c_{\alpha'}$. 
This proves that $p$ forces that $\dot c_\alpha \cap \xi \subseteq \dot c_{\alpha'}$. 
The other inclusion is proved using a symmetric argument.
\end{proof}

Let us show that $\Box_{\omega_1}$ holds in any generic extension by $\p$. 
This follows from well-known arguments which we review for completeness. 
First note that it suffices to find a sequence 
$\langle d_\alpha : \alpha \in \omega_2 \cap \cof(\omega_1) \rangle$ 
such that each $d_\alpha$ is a club subset of $\alpha$ with order 
type $\omega_1$, and for any $\alpha < \alpha'$ and $\xi$ a common limit 
point of $d_\alpha$ and $d_{\alpha'}$, 
$d_\alpha \cap \xi = d_{\alpha'} \cap \xi$. 
For then we can extend this sequence to a square sequence by defining 
$d_\gamma$ for $\gamma \in \omega_2 \cap \cof(\omega)$ by letting 
$d_\gamma = d_\alpha \cap \gamma$ for some (any) $\alpha$ in 
$\omega_2 \cap \cof(\omega_1)$ such that $\gamma$ is a limit point of 
$d_\alpha$, and if no such $\alpha$ exist, letting 
$d_\gamma$ be a cofinal subset 
of $\gamma$ of order type $\omega$.

Recall that each $\alpha$ in $\Lambda$ is in $C^* \cap \cof(\omega_1)$ and 
is a limit point of $C^*$. 
For each $\alpha \in \Lambda$ let 
$d_\alpha = \lim(c_\alpha) \cap C^* \cap \alpha$. 
Then by Lemma 4.5 and Proposition 4.7, the sequence 
$\langle d_\alpha : \alpha \in \Lambda \rangle$ satisfies that each 
$d_\alpha$ is a club subset of $\alpha$ with order type $\omega_1$, and for all 
$\xi$ in $d_\alpha \cap d_{\alpha'}$, 
$d_\alpha \cap \xi = d_{\alpha'} \cap \xi$.

One can easily prove by induction that for any $\xi < \omega_2$, 
there exists a sequence $\langle e_\beta : \beta \in \xi \cap \cof(\omega_1) \rangle$ 
such that each $e_\beta$ is a club subset of $\beta$ of order type $\omega_1$ 
and any $e_\beta$ and $e_{\beta'}$ share no common limit points. 
Consider $\beta_0 < \beta_1$ which are 
consecutive elements of $C^* \cup \{ 0 \}$. 
Using the fact just mentioned, we can transfer a sequence of clubs defined 
on $\ot(\beta_1 \setminus \beta_0) \cap \cof(\omega_1)$ to a sequence 
$\langle d_\alpha : \alpha \in (\beta_0,\beta_1) \cap 
\cof(\omega_1) \rangle$ so that each $d_\alpha$ is a club subset of $\alpha$ 
with minimum element greater than $\beta_0$ and order type $\omega_1$, 
such that any $d_{\alpha}$ and $d_{\alpha'}$ share no common limit points. 
But any ordinal in $\omega_2 \cap \cof(\omega_1)$ which is not in $C^*$ 
belongs to such an interval. 
So we have defined $d_\alpha$ for all $\alpha \in \omega_2 \cap \cof(\omega_1)$. 
It is straightforward to check that the extended sequence is as required.

\bibliographystyle{plain}
\bibliography{paper23}

\end{document}